\documentclass{amsart}
\usepackage[english]{babel}
\usepackage{color,latexsym,amsfonts,amsmath,enumerate,graphics,amsthm,amssymb}

\title[Hilbert and Thompson geometries]
{Hilbert and Thompson geometries isometric to infinite-dimensional
Banach spaces}
\author{Cormac Walsh}
\address{Inria, CMAP, \'{E}cole Polytechnique, Universit\'e Paris-Saclay,
91128 Palaiseau, France}
\email{cormac.walsh@inria.fr}

\date{\today}

\newcommand\thompson{d_T}
\newcommand\dhil{d_H}
\newcommand\closure{\operatorname{cl}}
\newcommand\after{\circ}
\newcommand\intersection{\cap} 

\newcommand\union{\cup}

\newcommand\R{\mathbb{R}}
\newcommand\Z{\mathbb{Z}}
\newcommand\N{\mathbb{N}}

\newcommand\projC{P(C)}
\newcommand\myord{\preceq}

\newcommand\dotprod[2]{\langle #1, #2 \rangle}
\newcommand\indicator{\operatorname{I}}

\newcommand\compactum{K}
\newcommand\ca{\signedmeasures(\compactum)}
\newcommand\CK{C(K)}

\newcommand\gauge[3]{M_{#1}({#2},{#3})}

\newcommand\tmetric{d_T}
\newcommand\detourmetric{\delta}

\newcommand\Raug{\overline\R}
\newcommand\compat{\sim}

\let\epsilon=\varepsilon
\let\phi=\varphi
\let\theta=\vartheta
\newtheorem{theorem}{Theorem}[section]
\newtheorem{Definition}[theorem]{Definition}
\newtheorem{questions}[theorem]{Questions}
\newtheorem{lemma}[theorem]{Lemma}
\newcounter{mycounter}
\newtheorem{claim}[mycounter]{Claim}
\newtheorem{corollary}[theorem]{Corollary}
\newtheorem{proposition}[theorem]{Proposition}

\newenvironment{definition}{\begin{Definition}\rm}{\end{Definition}}
\renewcommand{\theenumi}{(\roman{enumi})}

\newtheorem*{repeatedtheorem1}{Theorem~\ref{thm:hilbert_banach}}
\newtheorem*{repeatedtheorem2}{Theorem~\ref{thm:thompson_banach}}

\begin{document}

\begin{abstract}
We study the horofunction boundaries of Hilbert and Thompson geometries,
and of Banach spaces, in arbitrary dimension.
By comparing the boundaries of these spaces,
we show that the only Hilbert and Thompson geometries that are isometric
to Banach spaces are the ones defined on the cone of positive continuous
functions on a compact space.
\end{abstract}
\maketitle

\newcommand\directed{\mathcal{D}}
\newcommand\newdirected{\mathcal{D'}}
\newcommand\neighbourhoods{\mathcal{N}}

\newcommand\firstdirected{\mathcal{D}}
\newcommand\seconddirected{\mathcal{D'}}
\newcommand\newfirstdirected{\mathcal{E}}
\newcommand\newseconddirected{\mathcal{E'}}
\newcommand\combdirected{\mathcal{Y}}

\newcommand\unint{K}
\newcommand\uncont{C(K)]}
\newcommand\uncontpos{C^+(K)}

\newcommand\dee{d}
\newcommand\signedmeasures{ca_r}

\newcommand\drevfunk{d_{R}}
\newcommand\dfunk{d_{F}}

\section{Introduction}
\label{sec:introduction}

It was observed by Nussbaum~\cite[pages 22--23]{nussbaum:hilbert}
and de la Harpe \cite{delaharpe} that the Hilbert geometry on a
finite-dimensional simplex is isometric to a normed space. Later, Foertsch and
Karlsson~\cite{foertsch_karlsson_hilbert_metrics_and_minkowski_norms}
proved the converse, that is, if a Hilbert geometry on a
finite-dimensional convex domain is isometric to a normed space,
then the domain is a simplex.

In this paper, we extend this result to infinite dimension.

The natural setting is that of order unit spaces.
An order unit space is a triple $(X, \overline C, u)$ consisting of
a vector space $X$, an Archimedean convex cone $\overline C$ in $X$,
and an order unit $u$.
Let $C$ be the interior of $\overline C$ with respect to the topology on $X$
coming from the order unit norm.
Define, for each $x$ and $y$ in $C$,
\begin{align*}
\gauge{}{x}{y} := \inf\{ \lambda >0 \mid x \le \lambda y \}.
\end{align*}
Since every element of $C$ is an order unit, this quantity is finite.
Hilbert's projective metric is defined to be
\begin{align*}
\dhil(x,y) := \log \gauge{}{x}{y} \gauge{}{y}{x},
\qquad\text{for each $x, y \in C$}.
\end{align*}
It satisfies $\dhil(\lambda x, \nu y) = \dhil(x,y)$,
for all $x,y\in C$ and $\lambda, \nu>0$,
and is a metric on the projective space $\projC$ of the cone.

In infinite dimension the role of the simplex will be played by the cone
$\uncontpos$ of positive continuous functions on a
compact topological space $K$.
This cone lives in the linear space $C(K)$ of continuous functions on $K$,
and is the interior of $\overline\uncontpos$, the cone of non-negative
continuous functions on $K$.
The triple $(C(K), \overline\uncontpos, u)$ forms an order unit space, where
$u$ is the function on $K$ that is identically $1$.

It is not hard to show that the Hilbert metric on $\uncontpos$ is the
following:
\begin{align*}
\dhil(x,y) = \log \sup_{k,k'\in K} \frac{x(k)}{y(k)} \frac{y(k')}{x(k')},
\qquad\text{for $x, y\in \uncontpos$}.
\end{align*}
The map $\log \colon \uncontpos \to C(K)$ that takes the logarithm
coordinate-wise is an isometry when $C(K)$ is equipped with the semi-norm
\begin{align*}
||z||_H :=\sup_{k\in K} z_k - \inf_{k\in K}z_{k}.
\end{align*}
Denote by ${\equiv}$ the equivalence relation on $C(K)$ where two functions
are equivalent if they differ by a constant, that it, $x \equiv y$
if $x = y + c$ for some constant $c$.
The seminorm $||\cdot||_H$ is a norm on the quotient space $C(K) / {\equiv}$.
This space is a Banach space, and we denote it by $H(K)$.
The coordinate-wise logarithm map induces an isometry from the projective space
$P(C(K))$ to $H(K)$.

We show that every Hilbert geometry isometric to a Banach space arises
in this way.
When we talk about the Hilbert geometry on a cone $C$,
we assume that $C$ is the interior of the cone of some order unit space,
and we mean the Hilbert metric on the projective space $P(C)$.

\begin{repeatedtheorem1}
If a Hilbert geometry on a cone $C$ is isometric to a Banach space,
then $C$ is linearly isomorphic to $\uncontpos$, for some compact
Hausdorff space $K$.
\end{repeatedtheorem1}

We also prove a similar result for another metric related to the Hilbert
metric, the Thompson metric. This is defined, on the interior $C$ of the
cone of an order unit space, in the following way:
\begin{align*}
\thompson(x,y) := \log \max \big(\gauge{}{x}{y}, \log \gauge{}{x}{y} \big),
\qquad\text{for each $x, y \in C$}.
\end{align*}
Note that the Thompson metric is a metric on $C$, not on its projective space.

\begin{repeatedtheorem2}
If a Thompson geometry on a cone $C$ is isometric to a Banach space,
then $C$ is linearly isomorphic to $\uncontpos$,
for some compact Hausdorff space $K$.
\end{repeatedtheorem2}

The main technique we use in both cases is to compare the horofunction boundary
of the Banach space with that of the Hilbert or Thompson geometry.
The horofunction boundary was first introduced by
Gromov~\cite{gromov:hyperbolicmanifolds}.
Since it is defined purely in terms of the metric structure,
it is useful for studying isometries of metric spaces.

In finite dimension, the horofunction boundary of normed spaces was
investigated in~\cite{walsh_normed} and that of Hilbert geometries
in~\cite{walsh_hilbert}.
The results for the Hilbert geometry were used to determine the isometry
group of this geometry in the polyhedral case~\cite{lemmens_walsh_polyhedral}
and in general~\cite{walsh_gauge}.
See~\cite{walsh_hilbert_survey} for a survey of these results.
Other papers have dealt with isometries of Hilbert geometries in special cases.
See~\cite{delaharpe} for simplices and strictly convex domains;
\cite{troyanov_matveev_two_dimensional} for the two dimensional case;
\cite{molnar_thompson_isometries} and~\cite{molnar_2d} for the set of
positive definite Hermitian matrices of trace $1$;
and~\cite{bosche} for the cross-section of a symmetric cone.
Many partial results are contained in the thesis~\cite{speer_thesis}.
There are also some results concerning isometries in infinite
dimension~\cite{lemmens_roelands_wortel, lemmens_roelands_wortel_jb_algebras}.

Usually, when one develops the theory of the horofunction boundary,
one makes the assumption that the space is proper, that is,
that closed balls are compact.
For normed spaces and Hilbert geometries, this is equivalent to the dimension
being finite.
To deal with infinite-dimensional spaces,
we are forced to extend the framework. For example, we must use nets rather
than sequences. In section~\ref{sec:preliminaries},
we reprove some basic results concerning the horofunction boundary
in this setting.
We study the boundary of normed spaces in section~\ref{sec:normed_spaces}.
In this and later sections, we make extensive use of the theory of affine
functions on a compact set, including some Choquet Theory.
To demonstrate the usefulness of the horofunction boundary, we give a short
proof of the Masur--Ulam theorem in section~\ref{sec:masur_ulam}.
We determine explicitly the Busemann points in the boundary
of the important Banach space $(C(K), ||\cdot||_\infty)$
in section~\ref{sec:sup_space}.
Crucial to our method will be to consider the Hilbert and Thompson metrics as
symmetrisations of a non-symmetric metric, the \emph{Funk metric}.
In sections~\ref{sec:reverse_funk}, \ref{sec:funk}, and~\ref{sec:hilbert},
we study the boundaries of, respectively, the reverse of the Funk metric,
the Funk metric itself, and the Hilbert metric.
Again, we take a closer look, in section~\ref{sec:positive_cone}, at an example,
here the cone $\uncontpos$. We study the boundary of the Thompson geometry
in section~\ref{sec:thompson}, which allows us to prove
Theorem~\ref{thm:thompson_banach} in section~\ref{sec:thompson_isometries}.
We prove Theorem~\ref{thm:hilbert_banach} in
section~\ref{sec:hilbert_isometries}.

\emph{Acknowledgements.}
I had many very useful discussions with Bas Lemmens about this work.
This work was partially supported by the ANR `Finsler'.

\section{Preliminaries}
\label{sec:preliminaries}

\subsection{Hilbert's metric}

\newcommand\linspace{X}

Let $\overline C$ be a cone in a real vector space $X$. In other words,
$\overline C$ is closed under addition and multiplication by non-negative
real numbers, and $\overline C\intersection -\overline C = 0$.
The cone $\overline C$ induces a partial ordering $\le$ on $X$ by $x\le y$
if $y-x\in \overline C$. We say that $\overline C$ is \emph{Archimedean} if,
whenever $x\in X$ and $y\in \overline C$
satisfy $nx \le y$ for all $n\in\N$, we have $x\le 0$.
An order unit is an element $u\in \overline C$ such that for each $x\in X$
there is some $\lambda > 0$ such that $x\le \lambda u$.
An order unit space $(X, \overline C, u)$ is a vector space $X$ equipped
with a cone $\overline C$ containing an order unit $u$.

We define the \emph{order unit norm} on $X$:
\begin{align*}
||x||_u := \inf \{ \lambda>0  \mid -\lambda u \le x \le \lambda u \},
\qquad\text{for all $x\in X$}.
\end{align*}
We use on $X$ the topology induced by $||\cdot||_u$.
It is known that, under this topology, $\overline C$ is
closed~\cite[Theorem~2.55]{aliprantis_charalambos_tourky_cones_and_duality}
and has non-empty interior. Indeed, the interior $C$ of $\overline C$
is precisely the set of its order units; see~\cite{lemmens_roelands_wortel}.

On $C$ we define Hilbert's projective metric as in the introduction.

Hilbert originally defined his metric on bounded open convex sets.
Suppose $\Omega$ is such a set, and that we are given two distinct points
$x$ and $y$ in $\Omega$.
Define $w$ and $z$ to be the points in the boundary $\partial \Omega$
of $\Omega$ such that $w$, $x$, $y$, and $z$ are collinear and arranged in this
order along the line in which they lie.
Hilbert defined the distance between $x$ and $y$ to be
the logarithm of the cross ratio of these four points:
\begin{equation*}
\dhil(x,y):= \log \frac{|xz|\,|wy|}{|yz|\,|wx|}.
\end{equation*}

In the case of an infinite-dimensional cone, one may recover Hilbert's
original definition if the cone has a strictly positive state,
that is, if there exists a continuous linear functional $\psi$
that is positive everywhere on $C$. In this situation, Hilbert's definition
applied to the cross section $\Omega := \{x\in C\mid \psi(x) = 1\}$ agrees with
the definition in the introduction.
It was shown in~\cite{lemmens_roelands_wortel} that Hilbert's definition
makes sense if and only if the convex domain is affinely isomorphic to
a cross section of the cone of an order unit space.

Not every order unit space has a strictly positive state however.
For example, take $X := \R^Y$, the space of real-valued bounded functions
on some uncountable set $Y$. The subset of these functions that are
non-negative is an Archimedean cone,
and the function that is identically $1$ is an order unit.
On spaces such as these, the metric $\dhil$ is well-defined even though
Hilbert's original construction is not.

\begin{figure}
\input{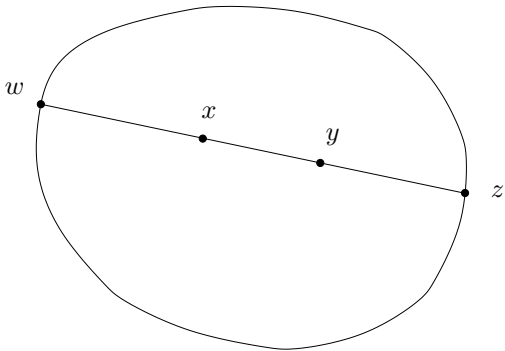}
\caption{Hilbert's definition of a distance.}
\label{fig:hilbert_def}
\end{figure}

\subsection{The Funk and reverse-Funk metrics}

Essential to our method will be to consider the Hilbert and Thompson metrics as
symmetrisations of the \emph{Funk metric}, which is defined as follows:
\begin{align*}
\dfunk(x,y) := \log\gauge{}{x}{y},
\qquad\text{for all $x\in \linspace$ and $y\in C$}.
\end{align*}
This metric first appeared in~\cite{funk}.
We call its reverse $\drevfunk(x,y):=\dfunk(y,x)$
the \emph{reverse-Funk metric}.

Like Hilbert's metric, the Funk metric was first defined on bounded open
convex sets. On a cross section $D$ of a cone $C$, one can show that
\begin{align*}
\dfunk(x,y) = \log\frac{|xz|}{|yz|}
\qquad \text{and} \qquad
\drevfunk(x,y) = \log\frac{|wy|}{|wx|},
\end{align*}
for all $x,y\in D$. Here $w$ and $z$ are the points of the boundary
$\partial D$ shown in Figure~\ref{fig:hilbert_def}.

On $D$, the Funk metric is a \emph{quasi-metric}, in other words, it satisfies
the usual metric space axioms except that of symmetry.
On $C$, it satisfies the triangle inequality but is not
non-negative. It has the following homogeneity property:
\begin{align*}
\dfunk(\alpha x, \beta y) = \dfunk(x,y) + \log\alpha - \log\beta,
\qquad\text{for all $x,y\in C$ and $\alpha,\beta>0$}.
\end{align*}

Observe that both the Hilbert and Thompson metrics are symmetrisations
of the Funk metric:
for all $x,y\in C$,
\begin{align*}
\dhil(x,y) &= \dfunk(x,y) + \drevfunk(x,y)
\qquad\text{and} \\
\thompson(x,y) &= \max\big(\dfunk(x,y), \drevfunk(x,y)\big).
\end{align*}

\subsection{The horofunction boundary}

\newcommand\iso{f}
\newcommand\mapclass{h}
\newcommand\distfn{\psi}
\newcommand\dist{d}
\newcommand\symdist{d_{\text{sym}}}
\newcommand\geo{\gamma}
\newcommand\horofunction{h}
\newcommand\isometric{\cong}

Let $(X,\dist)$ be a metric space.
Associate to each point $z\in X$ the function $\distfn_z\colon X\to \R$,
\begin{equation*}
\distfn_z(x) := \dist(x,z)-\dist(b,z),
\end{equation*}
where $b\in X$ is some base-point. It can be shown that the map
$\distfn\colon X\to C(X),\, z\mapsto \distfn_z$ is injective and continuous.
Here, $C(X)$ is the space of continuous real-valued functions
on $X$, with the topology of uniform convergence on compact sets.
We identify $X$ with its image under $\distfn$.

Let $\closure$ denote the topological closure operator.
Since elements of $\closure\distfn(X)$ are equi-Lipschitzian,
uniform convergence on compact sets is equivalent to pointwise convergence,
by the Ascoli--Arzel\`a theorem.
Also, from the same theorem, the set $\closure\distfn(X)$ is compact.
We call it the \emph{horofunction compactification}.
We define the \emph{horofunction boundary} of $(X,\dist)$ to be
\begin{align*}
X(\infty):=\big(\closure\distfn(X)\big)\backslash\distfn(X),
\end{align*}
The elements of this
set are the \emph{horofunctions} of $(X,\dist)$.

Although this definition appears to depend on the choice of base-point,
one may verify that horofunction boundaries coming from different base-points
are homeomorphic, and that corresponding horofunctions differ only by an
additive constant.

\subsection{Almost geodesics and Busemann points}

In the finite-dimensional setting one commonly considers geodesics
parameterised by $\Z$ or $\R$. In infinite dimension, however,
one must use nets.
Recall that a directed set is a nonempty pre-ordered set such that
every pair of elements has an upper bound in the set,
and that a net in a topological space is a function from a directed set to the
space.

\begin{definition}
\label{def:almost_non_increasing}
A net of real-valued functions $f_\alpha$ is \emph{almost non-increasing} if,
for any $\epsilon>0$, there exists $A$ such that
$f_\alpha \ge f_{\alpha'} - \epsilon$, for all $\alpha$ and $\alpha'$ greater
than $A$, with $\alpha\le\alpha'$.
\end{definition}

An \emph{almost non-decreasing} net is defined similarly.

Observe that if $f_\alpha$ is an almost non-increasing net of functions,
and $m_\alpha$ is a net (on the same directed set) of real numbers converging
to zero, then $f_\alpha + m_\alpha$ is also almost non-increasing.

\begin{definition}
\label{def:almost_almost_geodesic}
A net in a metric space is~\emph{almost geodesic} if, for all $\epsilon>0$,
\begin{align*}
d(b,z_{\alpha'}) \ge d(b,z_\alpha) + d(z_\alpha,z_{\alpha'}) - \epsilon,
\end{align*}
for $\alpha$ and $\alpha'$ large enough, with $\alpha\le\alpha'$.
\end{definition}

This definition is similar to Rieffel's~\cite{rieffel_group},
except that here we use nets rather than sequences and the almost geodesics
are unparameterised.
Note that any subnet of an almost geodesic is an almost geodesic.
Almost geodesics are very closely related to almost non-increasing nets,
as the following proposition shows.

\begin{proposition}
\label{prop:almost_non_inc_busemann}
Let $z_\alpha$ be a net in a metric space.
Then, $z_\alpha$ is an almost geodesic if and only if
$\phi_{z_\alpha} := d(\cdot,z_\alpha)-d(b,z_\alpha)$ is an
almost non-increasing net.
\end{proposition}

\begin{proof}
Let $\epsilon>0$ be given.
Assume $z_\alpha$ is almost geodesic. So, for $\alpha$ and $\alpha'$ large
enough, with $\alpha\le\alpha'$,
\begin{align*}
d(b,z_{\alpha'}) \ge d(b,z_\alpha) + d(z_\alpha,z_{\alpha'}) - \epsilon.
\end{align*}
Let $x$ be a point in the metric space.
Combining the above inequality with the triangle inequality concerning the
points $x$, $z_\alpha$, and $z_{\alpha'}$, we get
\begin{align}
\label{eqn:phi_non_inc}
d(x,z_\alpha) - d(b,z_\alpha)
   \ge d(x,z_{\alpha'}) - d(b,z_{\alpha'}) - \epsilon.
\end{align}
We conclude that the net $\phi_{z_\alpha}$ is almost non-increasing.

Now assume that $\phi_{z_\alpha}$ is almost non-increasing,
in other words that~(\ref{eqn:phi_non_inc}) holds
when $\alpha$ and $\alpha'$ are large enough, with $\alpha\le\alpha'$,
for all points $x$.
Taking $x$ equal to $z_\alpha$, we get that $z_\alpha$ is an almost geodesic.
\end{proof}

The following lemma extends Dini's theorem.

\newcommand\dinispace{X}

\begin{lemma}
\label{lem:improved_convergence_of_sups}
Let $g_\alpha$ be an almost non-increasing net of functions on a Hausdorff
space~$\dinispace$.
Then, $g_\alpha$ converges pointwise to a function $g$.
If the $g_\alpha$ are upper-semicontinuous, then so is the limit.
If furthermore $\dinispace$ is compact, then $\sup g_\alpha$ converges
to $\sup g$.
\end{lemma}

\begin{proof}
Let $g_\alpha$ be an almost non-increasing net of functions on $\dinispace$,
and choose $x\in \dinispace$. It is clear that for each $\epsilon>0$, we have
$\liminf_\alpha g_\alpha(x) \ge \limsup_\alpha g_\alpha(x) - \epsilon$,
from which it follows that $g_\alpha(x)$ converges.
Denote by $g$ the pointwise limit of $g_\alpha$.

Assume that each $g_\alpha$ is upper semicontinuous.
Let $x_\beta$ be a net in $\dinispace$ converging to a point $x\in \dinispace$.
So $g_\alpha(x) \ge \limsup_\beta g_\alpha(x_\beta)$, for each $\alpha$.
Let $\epsilon>0$.
That $g_\alpha$ is almost non-increasing implies that
$g\le g_\alpha + \epsilon$ for $\alpha$ large enough.
So, choose $\alpha$ large enough that this holds and
$g_\alpha(x) \le g(x) + \epsilon$.
Putting all this together, we get
$g(x) \ge \limsup_\beta g(x_\beta) - 2\epsilon$, and we conclude that
$g$ is upper semicontinuous.

Now assume that $\dinispace$ is compact.
So, for each $\alpha$, since $g_\alpha$ is upper semicontinuous,
it attains its supremum at some point $x_\alpha$,
and furthermore the net $x_\alpha$ has a cluster point $x$ in $\dinispace$.
By passing to a subnet if necessary, we may assume that $g_\alpha(x_\alpha)$
converges to a limit $l$, and that $x_\alpha$ converges to $x$.

Let $\epsilon>0$.
For $\alpha$ large enough,
$g_{\alpha'}(x_{\alpha'}) \le g_\alpha(x_{\alpha'}) + \epsilon$
for all $\alpha' \ge \alpha$.
Taking the limit supremum in $\alpha'$,
using the upper semicontinuity of $g_\alpha$,
and then taking the limit in $\alpha$ we get that
$l \le g(x) + \epsilon \le \sup g + \epsilon$, and hence that
$l \le \sup g$, since $\epsilon$ was chosen arbitrarily.
The opposite inequality comes from the fact that, in general,
the limit of a supremum is greater than or equal to the supremum of the limit.
We have shown that any limit point of $\sup g_\alpha = g_\alpha(x_\alpha)$
is equal to $\sup g$.
\end{proof}

It is clear from Proposition~\ref{prop:almost_non_inc_busemann}
and the first part of Lemma~\ref{lem:improved_convergence_of_sups}
that every almost geodesic converges in the horofunction
compactification $X\union X(\infty)$.
We say that horofunction, in other words an element of $X(\infty)$,
is a \emph{Busemann point} if it is the limit of an almost geodesic.

The following proposition will be needed in section~\ref{sec:hilbert}.

\begin{proposition}
\label{prop:almost_geos_are_infinite}
Let $z_\alpha$ be an almost geodesic in a complete metric space $(X, d)$
with basepoint $b$. If $d(b,z_\alpha)$ is bounded, then $z_\alpha$
converges to a point in $X$.
\end{proposition}

\begin{proof}
For any $\epsilon>0$,
we have, for $\alpha$ and $\alpha'$ large enough with $\alpha \le \alpha'$,
\begin{align*}
d(b, z_\alpha)
    \le d(b, z_\alpha) + d(z_\alpha, z_{\alpha'})
    \le d(b, z_{\alpha'}) + \epsilon.
\end{align*}
Hence, $d(b, z_\alpha)$ is an almost non-decreasing net of real numbers.
By assumption, it is bounded above. We deduce that it converges to some
real number, as $\alpha$ tends to infinity.

Using this and again the almost-geodesic property of the net $z_\alpha$,
we get that, given any $\epsilon>0$,
we have $d(z_\alpha, z_{\beta}) < \epsilon$, for all $\alpha$ and $\beta$
large enough with $\alpha \le \beta$.
The same is true when $\alpha'$ is substituted for $\alpha$.
It follows using the trianlge inequality that
$d(z_\alpha, z_{\alpha'}) < 2 \epsilon$, for $\alpha$ and $\alpha'$
large enough, without requiring that $\alpha \le \alpha'$.
So, we see that $z_\alpha$ is a Cauchy net, and hence converges since $X$
is assumed complete.
\end{proof}

\subsection{The detour cost}

Let $(X,d)$ be a metric space with base-point $b$.
One defines the \emph{detour cost} for any two horofunctions $\xi$
and $\eta$ in $X(\infty)$ to be
\begin{align*}
H(\xi,\eta)
   &:= \sup_{W\ni\xi} \inf_{x\in W\intersection X} \big( d(b,x)+\eta(x) \big),
\end{align*}
where the supremum is taken over all neighbourhoods $W$ of $\xi$ in
$X\cup X(\infty)$.
This concept first appeared in~\cite{AGW-m} in a slightly different setting.
More detail about it can be found in~\cite{walsh_stretch}.

\newcommand\mynet{w}
\newcommand\myindex{\beta}
\newcommand\subnet{z}
\newcommand\subindex{\alpha}

\begin{lemma}
\label{lem:exists_net}
Let $\xi$ and $\eta$ be horofunctions of a metric space $(X,d)$.
Then, there exists a net $\subnet_\subindex$ converging to $\xi$ such that
\begin{equation*}
H(\xi, \eta)
   = \lim_{\myindex}
        \big( d(b, \subnet_\subindex) + \eta(\subnet_\subindex) \big).
\end{equation*}
\end{lemma}

\begin{proof}
To ease notation, write $f(x) := d(b, x) + \eta(x)$, for all $x\in X$.

Let $\neighbourhoods$ be the set of neighbourhoods of $\xi$ in
$X\union X(\infty)$.
Define a pre-order on the set
\begin{align*}
\directed := \{(W, x) \in \neighbourhoods \times X
                      \mid \text{$x\in W \intersection X$} \}
\end{align*}
by $(W_1, x_1) \le (W_2, x_2)$ if $W_1 \supset W_2$.
This pre-order makes $\directed$ into a directed set.

For each $\myindex := (W, x) \in \directed$, let $\mynet_\myindex := x$.
Clearly, the net $\mynet_\myindex$ converges to $\xi$.

Let $E$ be an open neighbourhood of $H(\xi,\eta)$ in $[0, \infty]$,
and let $(W', x') \in \directed$.
Take $W\in\neighbourhoods$ small enough that $W \subset W'$ and
\begin{align*}
\inf_{x\in W\intersection X} f(x) \in E.
\end{align*}
We can then take $x\in W \intersection X$ such that
$f(x) \in E$.
So, $\myindex := (W, x)$ satisfies $\myindex \ge (W', x')$ and
$f(\mynet_\myindex) \in E$.
This shows that $H(\xi,\eta)$ is a cluster point of the net
$f(\mynet_\myindex)$.

Therefore, there is some subnet $\subnet_\subindex$ of $\mynet_\myindex$
such that $f(\subnet_\subindex)$ converges to $H(\xi,\eta)$.
\end{proof}

The following was proved in~\cite[Lemma 3.3]{walsh_minimum} in a slightly
different setting. There is a proof in~\cite{walsh_stretch} that works with
very little modification in the present setting with nets.

\begin{lemma}
\label{lem:along_geos}
Let $z_\alpha$ be an almost-geodesic net converging to a Busemann point $\xi$,
and let $y\in X$.
Then,
\begin{equation*}
\lim_{\alpha} \big( d(y, z_\alpha) + \xi(z_\alpha) \big) = \xi(y).
\end{equation*}
Moreover, for any horofunction $\eta$,
\begin{equation*}
 H(\xi,\eta) = \lim_{\alpha} \big( d(b, z_\alpha) + \eta(z_\alpha) \big).
\end{equation*}
\end{lemma}

The proof of the next result however is different for nets.

\begin{theorem}
\label{thm:equivalence_busemann}
A horofunction $\xi$ is a Busemann point if and only if $H(\xi, \xi) = 0$.
\end{theorem}

\begin{proof}
If $\xi$ is a Busemann point, then it follows from Lemma~\ref{lem:along_geos}
that $H(\xi, \xi) = 0$.

Now assume that $\xi$ is a horofunction satisfying $H(\xi, \xi) = 0$.
By Lemma~\ref{lem:exists_net}, there is a net $z_\alpha: \directed \to X$
in the metric space $(X, d)$
converging to $\xi$ such that $d(b,z_\alpha) + \xi(z_\alpha)$
converges to zero.

Define the set
\begin{multline*}
\newdirected := \Big\{
     (\alpha, \beta, \epsilon) \in \directed \times \directed \times (0,\infty)
          \mid \\
            \text{$\alpha \le \beta$,
            and $|d(z_\alpha, z_{\gamma}) - d(b, z_{\gamma}) - \xi(z_\alpha)|
                       < \epsilon$
                  for all $\gamma\ge\beta$}
\Big\}.
\end{multline*}
Observe that, for any $\alpha\in\directed$ and $\epsilon >0$,
there exists $\beta\in\directed$
such that $(\alpha, \beta, \epsilon)\in\newdirected$,
because $d(\cdot, z_{\gamma}) - d(b, z_{\gamma})$ converges pointwise to
$\xi(\cdot)$.

Define on $\newdirected$ the pre-order $\le$, where
$(\alpha, \beta, \epsilon) \le (\alpha', \beta', \epsilon')$
if either the two are identical,
or if $\beta \le \alpha'$ and $\epsilon \ge \epsilon'$.
This relation is easily seen to be reflexive and transitive.
Also, it is not hard to show that $\newdirected$ is directed by $\le$.

The map $h \colon (\alpha, \beta, \epsilon) \mapsto \alpha$ is monotone,
and its image is cofinal, that is, for any $\alpha'\in\directed$,
there exists $(\alpha, \beta, \epsilon)\in\newdirected$ such that
$h(\alpha, \beta, \epsilon) \ge \alpha'$.

So, the net $y_\kappa$ defined by $y_\kappa := z_{h(\kappa)}$,
for all $\kappa\in\newdirected$, is a subnet of $z_\alpha$.
In particular, it converges to $\xi$.
Moreover, $d(b, y_\kappa) + \xi(y_\kappa)$ converges to zero.

Let $\kappa:= (\alpha, \beta, \epsilon)$
and $\kappa':= (\alpha', \beta', \epsilon')$ be elements of $\newdirected$,
satisfying $\kappa \le \kappa'$.
So, $\alpha' \ge \beta$, which implies that
$|d(z_\alpha, z_{\alpha'}) - d(b, z_{\alpha'}) - \xi(z_\alpha)| < \epsilon$.

Hence, for any $\epsilon>0$,
\begin{align*}
d(y_\kappa, y_{\kappa'}) - d(b, y_{\kappa'}) &< \xi(y_\kappa) + \epsilon \\
&< - d(b, y_\kappa) + 2\epsilon,
\end{align*}
for $\kappa$ and $\kappa'$ large enough, with $\kappa\le\kappa'$.
This proves that $y_\kappa$ is an almost-geodesic.
\end{proof}

The detour cost satisfies the triangle inequality and is non-negative.
By symmetrising the detour cost, we obtain a metric on the set of Busemann
points:
\begin{align*}
\delta(\xi,\eta):= H(\xi,\eta) + H(\eta,\xi),
\qquad\text{for all Busemann points $\xi$ and $\eta$}.
\end{align*}
We call $\delta$ the \emph{detour metric}. It is possibly infinite valued,
so it actually an \emph{extended metric}.
One may partition the set of Busemann points into disjoint subsets
in such a way that $\delta(\xi,\eta)$ is finite if and only if $\xi$
and $\eta$ lie in the same subset. We call these subsets the \emph{parts}
of the horofunction boundary. Of particular interest are the parts that
consist of a single Busemann point, which are called the \emph{singleton}
parts.

The following expression for the detour cost will prove useful.
See~\cite[Prop.~4.5]{walsh_gauge}.

\begin{proposition}
\label{prop:detour}
Let $\xi$ be a Busemann point, and $\eta$ a horofunction of a metric space
$(X,d)$. Then,
\begin{align*}
H(\xi,\eta)
   = \sup_{x\in X}\big(\eta(x)-\xi(x)\big)
   = \inf\big\{\lambda\in\R \mid \eta(\cdot) \le \xi(\cdot) + \lambda \big\}.
\end{align*}
\end{proposition}

\section{The Busemann points of a normed space}
\label{sec:normed_spaces}

\newcommand\dualball{{B^\circ}}
\newcommand\contaffines{A(\compactum, X^*)}
\newcommand\contaffinesball{A(\dualball, X^*)}

In this section, we determine the Busemann points of an arbitrary normed space.

Let $K$ be a convex subset of a locally-convex topological vector space $E$.
A function $f: K \to (-\infty, \infty]$ is said to be affine if
\begin{align*}
f\big( (1 - \lambda) x + \lambda y \big) = (1 - \lambda) f(x) + \lambda f(y),
\end{align*}
for all $x,y\in K$ and $\lambda\in(0, 1)$.
We denote by $A(K,E)$ the set of affine functions on $K$ that are the
restrictions of continuous finite-valued affine functions on the whole of $E$.
The following
is~\cite[Cor.~I.1.4]{alfsen_compact_convex_sets_and_boundary_integrals}

\begin{lemma}
\label{lem:CK_approximate_from_below}
If $K$ is a compact convex subset of $E$ and $a \colon K\to(-\infty,\infty]$
is a lower semicontinuous affine function, then there is a non-decreasing
net in $A(K,E)$ converging pointwise to $a$.
\end{lemma}

Let $(X, ||\cdot||)$ be a normed space.
We denote by $B$ the unit ball of $X$, and by $\dualball$ the dual ball.
The topological dual space of $X$ is denoted by $X^*$,
and we take the weak$^*$ topology on this space.
The dual ball is compact in the weak$^*$ topology,
by the Banach--Alaoglu theorem.

Recall that the Legendre--Fenchel transform of a function $f$ on $X$
is defined to be
\begin{align*}
f^*(y) := \sup_{x\in X} \big( \dotprod{y}{x} - f(x) \big),
\qquad\text{for all $y\in X^*$}.
\end{align*}
Since it is a supremum of weak* continuous affine functions, $f^*$ is weak*
lower semi-continuous and convex.
One may also define the transform of a function $g$ on $X^*$ as follows:
\begin{align*}
g^*(x) := \sup_{y\in X^*} \big( \dotprod{y}{x} - g(y) \big),
\qquad\text{for all $x\in X$}.
\end{align*}

These maps are inverses of one another in the following sense.
A function taking values in $(-\infty, \infty]$ is said to be \emph{proper}
if it is not identically $\infty$.
Recall that a lower semicontinuous convex function is automatically
weakly lower semicontinuous.
Denoting by $\Gamma(X)$ the proper lower semicontinuous convex functions on $X$,
and by $\Gamma^*(X^*)$ the proper weak* lower semicontinuous convex functions
on $X^*$, we have
\begin{align*}
f^{**} = f \quad\text{for $f\in \Gamma(X)$}
\qquad\text{and}\qquad
g^{**} = g \quad\text{for $g\in \Gamma^*(X^*)$}.
\end{align*}

We use the notation $f|_G$ to denote the restriction of a function $f$
to a set $G$.
Also, we denote by $X_B(\infty)$ the set of Busemann points
of a metric space $X$.

\begin{theorem}
\label{thm:characterise_busemen}
Let $(X, ||\cdot||)$ be a normed space.
A function on $X$ is in $X \union X_B(\infty)$ if and only if
it is the Legendre--Fenchel transform of a function that is affine
on the dual ball, infinite outside the dual ball,
weak* lower semi-continuous, and has infimum $0$.
\end{theorem}

\begin{proof}
The Legendre--Fenchel transform of any function is automatically weak* lower
semi-continuous. Every Busemann point is $1$-Lipschitz, and so its transform
takes the value $\infty$ outside the dual ball. Since each Busemann point
takes the value $0$ at the origin, the transform has infimum $0$.
That the transform of a Busemann point must be affine on the dual ball
was proved in~\cite[Lemma~3.1]{walsh_normed}; the theorem is stated there
for finite dimensional spaces, but the proof works in infinite dimension as
well.

Now let $f$ be a real-valued function on $X$ such that its
transform $f^*$ has the properties stated.
By Lemma~\ref{lem:CK_approximate_from_below},
there exists a non-decreasing net $g_\alpha$ of elements of
$\contaffinesball$ that converges pointwise to $f^*$.
For each $\alpha$,
we may write $g_\alpha = \dotprod{\cdot}{z_\alpha} |_\dualball + c_\alpha$,
where $z_\alpha \in X$ and $c_\alpha \in \R$.

Let $m_\alpha := \inf g_\alpha$, for each $\alpha$.
So, $m_\alpha$ is a non-decreasing net of real numbers, and by
Lemma~\ref{lem:improved_convergence_of_sups} it converges to $\inf f^* = 0$.
It is not too hard to calculate that, for each $\alpha$, the transform of
$\phi_{z_\alpha}(\cdot) := ||z_\alpha - \cdot\, || - ||z_\alpha||$
is $\phi^*_{z_\alpha} = g_\alpha - m_\alpha$;
see~\cite{walsh_normed}.

The Legendre--Fenchel transform is order-reversing, and so the net
$(g_\alpha)^*$ is non-increasing.
So, by the observation after the Definition~\ref{def:almost_non_increasing},
$(g_\alpha)^* + m_\alpha$ is almost non-increasing. But
\begin{align*}
(g_\alpha)^* + m_\alpha = (g_\alpha - m_\alpha)^*
   = \phi_{z_\alpha},
\qquad\text{for all $\alpha$}.
\end{align*}
Therefore, by Proposition~\ref{prop:almost_non_inc_busemann},
$z_\alpha$ is an almost geodesic in $(X, ||\cdot||)$.

Let $x$ be a point in $X$. We have
\begin{align*}
\phi_{z_\alpha}(x)
    = m_\alpha + \sup_{y\in \dualball}
            \big( \langle y, x \rangle - g_\alpha(y) \big),
\qquad\text{for all $\alpha$}.
\end{align*}
Since $g_\alpha$ is non-decreasing, the net of functions
$\langle \cdot, x \rangle - g_\alpha(\cdot)$ is non-increasing.
So, by Lemma~\ref{lem:improved_convergence_of_sups},
its supremum over the dual ball $\dualball$ converges to the supremum
of the pointwise limit $\langle \cdot, x \rangle - f^*(\cdot)$. 
We deduce that $\phi_{z_\alpha}$ converges pointwise to $f$.
We have thus proved that $f$ is either a Busemann point or a point in the
horofunction compactification corresponding to an element of $X$.
\end{proof}

We now determine the detour metric on the boundary of a normed space.

\begin{theorem}
\label{thm:normed_detour_metric}
Let $\xi_1$ and $\xi_2$ be Busemann points of a normed space,
having Legendre--Fenchel transforms $g_1$ and $g_2$, respectively.
Then, the distance between them in the detour metric is
\begin{align*}
\detourmetric(\xi_1, \xi_2)
   = \sup_{y\in \dualball} \big(g_1(y) - g_2(y)\big)
       + \sup_{y\in \dualball} \big(g_2(y) - g_1(y)\big).
\end{align*}
(We are using here the convention that $+\infty - (+\infty) = -\infty$.)
\end{theorem}

\begin{proof}
By the properties of the Legendre--Fenchel transform, we have,
for any $\lambda \in \R$, that $\xi_2 \le \xi_1 + \lambda$
if and only if $g_2 \ge g_1 - \lambda$.
So, applying Proposition~\ref{prop:detour}, we get
\begin{align*}
H(\xi_1,\xi_2)
   = \inf\big\{\lambda\in\R \mid g_2 \ge g_1 - \lambda \big\}
   = \sup_{y\in\dualball} \big(g_1(y) - g_2(y)\big).
\end{align*}
The result is now obtained upon symmetrising.
\end{proof}

\begin{corollary}
\label{cor:in_same_part_of_normed_space}
Two Busemann points of a normed space are in the same part
if and only if their respective Legendre--Fenchel transforms $g_1$ and $g_2$
satisfy
\begin{align}
\label{eqn:bounded_distance}
g_1 - c \le g_2 \le g_1 + c,
\qquad\text{for some $c\in\R$}.
\end{align}
\end{corollary}

\begin{corollary}
\label{cor:singletons_of_normed_space}
A function $\xi$ is a singleton Busemann point of a normed space if and
only if it is an extreme point of the dual ball.
\end{corollary}

\begin{proof}
If $\xi_1$ is an extreme point of the dual ball, then its transform $g_1$
takes the value zero at $\xi_1$ and infinity everywhere else. Let $\xi_2$ be
another Busemann point in the same part, which implies that its transform $g_2$
satisfies~(\ref{eqn:bounded_distance}). So, $g_2$ is finite at $\xi_1$
and infinite everywhere else, and since,
by Theorem~\ref{thm:characterise_busemen}, it has infimum zero,
we get that $g_2 = g_1$. Hence $\xi_2$ is identical to $\xi_1$.

Now let $\xi_1$ be a Busemann point that is not an extreme point of the
dual ball, and let $g_1$ be its transform. Since $g_1$ is affine on the
dual ball, the set on which it is finite is an extreme set of this ball,
and therefore must contain at least two points, for otherwise $\xi_1$
would be an extreme point. Choose an element $x$ of the normed space
such  that $\dotprod{\cdot}{x}$ separates these two points,
that is, does not take the same value at the two points. 
The function
\begin{align*}
g_2 := g_1 + \dotprod{\cdot}{x} - \inf_{\dualball}(g_1 + \dotprod{\cdot}{x})
\end{align*}
satisfies all the conditions of Theorem~\ref{thm:characterise_busemen},
and so its transform $\xi_2$ is a Busemann point.
Moreover, $g_1$ and $g_2$ satisfy~(\ref{eqn:bounded_distance}), which implies
that $\xi_2$ is in the same part as $\xi_1$. 
But, by construction, $g_2$ differs from $g_1$, and so $\xi_2$ differs from
$\xi_1$.
\end{proof}

\section{The Masur--Ulam Theorem}
\label{sec:masur_ulam}

The techniques developed so far allow us to write a short proof of the
Masur--Ulam theorem.

Recall that, according to Corollary~\ref{cor:singletons_of_normed_space},
the singleton Busemann points of a normed space are exactly
the extreme points of the dual ball.
Recall also that any surjective isometry between metric spaces can be extended
to a homeomorphism between their horofunction boundaries,
which maps singletons to singletons.

\newcommand\masurfn{\Lambda}

\begin{theorem}[Masur--Ulam]
\label{thm:masur_ulam}
Let $\masurfn \colon X\to Y$ be a surjective isometry between two normed spaces.
Then, $\masurfn$ is affine.
\end{theorem}

\begin{proof}
It will suffice to assume that $\masurfn$ maps the origin of $X$ to the origin
of $Y$, and show that it is linear.
So, take $\alpha,\beta\in\R$ and $x,y\in X$.

Let $f$ be an extreme point of the dual ball of $Y$.
So, $f$ is a singleton of the horofunction boundary of $Y$.
Therefore, $f\after\masurfn$ is a singleton of the horofunction boundary of $X$,
and hence an extreme point of the dual ball of $X$, and hence linear.
So,
\begin{align*}
f \big(\masurfn(\alpha x + \beta y)\big)
    &= \alpha f(\masurfn(x)) + \beta f(\masurfn(y)) \\
    &= f(\alpha\masurfn(x) + \beta\masurfn(y)),
\end{align*}
Since this is true for every extreme point $f$ of the dual ball of $Y$,
we have $\masurfn(\alpha x + \beta y) = \alpha \masurfn(x) + \beta \masurfn(y)$.
\end{proof}

\section{The horofunction boundary of $(C(K), ||\cdot||_\infty)$}
\label{sec:sup_space}

In this section we look in more detail at the space $C(K)$ with the supremum
norm, where $\compactum$ is an arbitrary compact Hausdorff space.
Here we can describe explicitly the Busemann points.
We use $\vee$ to denote maximum, and $\wedge$ to denote minimum.

\begin{theorem}
\label{thm:CK_characterise_busemanns}
The Busemann points of $(C(K), ||\cdot||_\infty)$ are the functions of the
following form
\begin{align}
\label{eqn:busemanns_of_CK}
\Phi(g) := \sup_{x\in\compactum} \big(-u(x) - g(x)\big)
             \vee \sup_{x\in\compactum}\big(-v(x) + g(x)\big),
\qquad\text{for all $g\in \CK$},
\end{align}
where $u$ and $v$ are two lower-semicontinuous functions from
$\compactum\to [0, \infty]$,
such that $\inf u \wedge \inf v = 0$, and such that $u(x) \vee v(x) =\infty$
for all $x\in\compactum$.
\end{theorem}

The proof will use the characterisation in the previous section
of the Legendre--Fenchel transforms of the Busemann points of a normed space.
Recall that these were shown to be the functions that are
affine on the dual ball, infinite outside the dual ball,
weak* lower semi-continuous, and have infimum $0$.
We will identify all such functions on the dual space of $C(K)$.

Recall that the dual space of $C(K)$ is $\signedmeasures(\compactum)$,
the set of regular signed Borel measures on $\compactum$ of bounded variation.
Any element $\mu$ of $\signedmeasures(\compactum)$ can be
written $\mu = \mu^+ - \mu^-$, where $\mu^-$ and $\mu^+$ are non-negative
measures. This is called the Jordan decomposition.
The dual norm is the \emph{total variation norm}, which satisfies
$||\mu|| = ||\mu^-|| + ||\mu^+||$.

\begin{proposition}
\label{prop:CK_characterise_affines}
Consider a function $\Theta \colon \signedmeasures(\compactum) \to [0,\infty]$
that is not the restriction to the dual ball of a continuous affine function.
Then, $\Theta$ is affine on the dual ball, infinite outside the dual ball,
weak* lower semi-continuous, and has infimum $0$
if and only if it can be written
\begin{align}
\label{eqn:CK_affines}
\Theta(\mu) =
    \Xi(\mu) :=
        \begin{cases}
        +\infty,
            & ||\mu|| \neq 1; \\
\displaystyle
        \int u \,\dee \mu^- + \int v \,\dee \mu^+,
            & ||\mu|| = 1,
        \end{cases}
\end{align}
where $u$ and $v$ are as in the statement of
Theorem~\ref{thm:CK_characterise_busemanns}.
\end{proposition}

The proof of this proposition will require several lemmas.

\begin{lemma}
\label{lem:dual_is_lsc}
The function $\Xi$ in~(\ref{eqn:CK_affines}) is lower semicontinuous.
\end{lemma}

\begin{proof}
Let $\mu_\alpha$ be a net in $\ca$ converging in the weak* topology to
$\mu\in\ca$. We must show that $\liminf_\alpha \Xi(\mu_\alpha) \ge \Xi(\mu)$.
By taking a subnet if necessary, we may suppose that $\Xi(\mu_\alpha)$
converges to a limit, which we assume to be finite.
This implies that $||\mu_\alpha||= 1$, eventually.
Since the dual unit ball is compact, we may, by taking a further subnet if
necessary, assume that $\mu_\alpha^+$
and $\mu_\alpha^-$ converge, respectively, to non-negative measures
$\nu$ and $\nu'$. These measures satisfy $\mu = \mu^+ - \mu^- = \nu - \nu'$,
and, so, from the minimality property of the Jordan decomposition,
$\nu \ge \mu^+$ and $\nu' \ge \mu^-$.

Since $u$ and $v$ are lower-semicontinuous,
we get from the Portmanteau theorem that
\begin{align}
\liminf_\alpha \Xi(\mu_\alpha)
    &\ge \liminf_\alpha \Big( \int u \, \dee\mu^-_\alpha \Big)
            + \liminf_\alpha \Big( \int v \, \dee\mu^+_\alpha \Big) \notag \\
\label{eqn:lim_inf_Xi}
    &\ge \int u \, \dee\nu'  + \int v \, \dee\nu.
\end{align}

Consider the case where $\bar\mu := \nu - \mu^+ = \nu' - \mu^-$ is non-zero.
Since $u + v$ is identically infinity, either $\int u \, \dee\bar\mu$
or $\int v \, \dee\bar\mu$ must equal infinity. This implies that
the right-hand-side of~(\ref{eqn:lim_inf_Xi}) is equal to infinity.

On the other hand, if $\nu = \mu^+$ and $\nu' = \mu^-$, then
\begin{align*}
||\mu|| = ||\mu^+|| + ||\mu^-|| = ||\nu|| + ||\nu'|| = \lim_\alpha||\mu_\alpha|| = 1.
\end{align*}
So, in this case, the right-hand-side of~(\ref{eqn:lim_inf_Xi}) is equal to
$\Xi(\mu)$.
\end{proof}

\begin{lemma}
\label{lem:affine_positive_part}
Let $\mu_1$ and $\mu_2$ be in the closed unit ball of
$\signedmeasures(\compactum)$,
and let $\mu := (1-\lambda)\mu_1 + \lambda \mu_2$,
for some $\lambda\in (0,1)$. If $||\mu|| = 1$,
then $||\mu_1|| = ||\mu_2|| = 1$, and
$\mu^+ = (1-\lambda)\mu^+_1 + \lambda \mu^+_2$
and $\mu^- = (1-\lambda)\mu^-_1 + \lambda \mu^-_2$.
\end{lemma}

\begin{proof}
Observe that the functions $\nu\mapsto ||\nu^+||$ and $\nu\mapsto ||\nu^-||$
are both convex, and hence
\begin{align}
\label{eqn:mu_plus}
||\mu^+|| &\le (1-\lambda)||\mu^+_1|| + \lambda ||\mu^+_2||
\qquad\text{and} \\
\label{eqn:mu_minus}
||\mu^-|| &\le (1-\lambda)||\mu^-_1|| + \lambda ||\mu^-_2||.
\end{align}
Moreover, the sum of these two functions is $\nu\mapsto ||\nu||$.
Using that $||\mu_1|| \le 1$ and $||\mu_2|| \le 1$, and $||\mu|| = 1$,
we deduce that inequalities~(\ref{eqn:mu_plus}) and~(\ref{eqn:mu_minus})
are actually equalities.

Since $\nu\mapsto \nu^+$ is also convex,
we have $\mu^+ \le (1-\lambda)\mu^+_1 + \lambda \mu^+_2$. Combining
this with the equalities just established, we get that
$\mu^+ = (1-\lambda)\mu^+_1 + \lambda \mu^+_2$, since the norm is
additive on non-negative measures.
The equation involving $\mu^-$ is proved similarly.
\end{proof}

\begin{lemma}
\label{lem:dual_is_affine}
The function $\Xi$ is affine on the unit ball of $\signedmeasures(\compactum)$.
\end{lemma}

\begin{proof}

Let $\mu$, $\mu_1$, and $\mu_2$ be in the unit ball of
$\signedmeasures(\compactum)$, such that
$\mu = (1-\lambda)\mu_1 + \lambda \mu_2$, for some $\lambda\in(0,1)$.
We wish to show that $\Xi(\mu) = (1-\lambda)\Xi(\mu_1) + \lambda \Xi(\mu_2)$.

Consider the case where $||\mu|| = 1$.
By Lemma~\ref{lem:affine_positive_part}, $||\mu_1|| = ||\mu_2|| = 1$,
and $\mu^+ = (1-\lambda)\mu^+_1 + \lambda \mu^+_2$
and $\mu^- = (1-\lambda)\mu^-_1 + \lambda \mu^-_2$.
We deduce that the second case in the definition of $\Xi$ is the relevant one,
for each of $\mu$, $\mu_1$, and $\mu_2$, and furthermore that the affine
relation holds.

Now consider the case where $||\mu|| < 1$. So, $\Xi(\mu) = \infty$.
To prove the affine relation, we must show that either $\Xi(\mu_1)$
or $\Xi(\mu_2)$ is infinite.

Assume for the sake of contradiction that both $\Xi(\mu_1)$ and $\Xi(\mu_2)$
are finite.
Denote by $U$ and $V$ the subsets of $\compactum$ where, respectively, $u$ and
$v$ are finite. So, $U$ and $V$ are disjoint.
From the definition of $\Xi$, we see that $||\mu_1||=||\mu_2||=1$,
that $\mu^+_1$ and $\mu^+_2$ are concentrated on $V$,
and that $\mu^-_1$ and $\mu^-_2$ are concentrated on $U$.
It follows that
$\mu^+ = (1-\lambda)\mu^+_1 + \lambda \mu^+_2$ and
$\mu^- = (1-\lambda)\mu^-_1 + \lambda \mu^-_2$.
So, $||\mu|| = ||\mu^+|| + ||\mu^-|| = 1$, which gives a contradiction.
\end{proof}

The following result is due to Choquet;
see~\cite[Theorem~I.2.6]{alfsen_compact_convex_sets_and_boundary_integrals}.

\begin{theorem}
\label{thm:CK_barycenter_formula}
If $f$ is a real-valued affine function of the first Baire class on a
compact convex set $K$ in a locally-convex Hausdorff space, then $f$
is bounded and $f(x) = \int f \,\dee\mu$, where $\mu$ is any
probability measure on $K$ and $x$ is the barycenter of $\mu$.
\end{theorem}

We will need a version of Lebesgue's monotone convergence theorem
for \emph{nets} of functions.
The following was proved in~\cite[Proposition 2.13]{baranov_woracek}.

\begin{lemma}
\label{lem:CK_monotone_convergence_for_nets}
Let $X$ be a locally-compact and $\sigma$-compact Hausdorff space, and let
$\lambda$ be a positive Borel measure that is complete and regular and
satisfies $\lambda(K)<\infty$ for all compact sets $K\subset X$.

Let $I$ be a directed set, and let $f_i \colon  X\to[0,\infty]$, $i\in I$, be a
family of lower semicontinuous functions that is monotone non-decreasing.
Set $f(x):=\sup_{i\in I} f_i(x)$ for all $x\in X$. Then,
\begin{align*}
\int_X f \,\dee\lambda = \sup_{i\in I} \int_X f_i \,\dee\lambda.
\end{align*}
\end{lemma}

\newcommand\myfunc{\Xi}

\begin{lemma}
\label{lem:CK_properties_imply_form}
If a function $\Xi \colon  \signedmeasures(\compactum) \to [0,\infty]$
is affine on the dual ball, infinite outside the dual ball,
weak* lower semi-continuous, and has infimum $0$,
then it either can be written in the form~(\ref{eqn:CK_affines})
or is the restriction to the dual ball of a continuous affine function
on the dual space.
\end{lemma}

\begin{proof}
Denote by $\delta_x$ the measure consisting of an atom of mass one at
a point $x$. Define the functions
\begin{align*}
v \colon  \compactum \to [0, \infty], \qquad v(x) &:= \myfunc(\delta_x),
\qquad\text{and} \\
u \colon \compactum \to [0, \infty],  \qquad u(x) &:= \myfunc(-\delta_x).
\end{align*}
These two functions are non-negative because $\inf \myfunc = 0$.
Moreover, the dual ball is weak* compact, and so,
as a lower semicontinuous affine function, $\myfunc$ attains its
infimum over it at an extreme point.
Recall that, in the present case, the extreme points are exactly the positive
and negative Dirac masses;
denote the set of these by $\partial_e := \partial_e^+ \union \partial_e^-$,
where $\partial_e^+ := \{\delta_x \mid x\in\compactum\}$
and $\partial_e^- := \{-\delta_x \mid x\in\compactum\}$.
Thus, $\inf u \wedge \inf v = 0$.

Also observe that $u$ and $v$ are lower semicontinuous.

Consider the case where $u(x)$ and $v(x)$ are both finite for some
$x\in\compactum$.
This implies that $\myfunc(0)$ is finite since $\myfunc$ is affine.
It follows from this that $\myfunc$ is finite on the whole of the dual ball.
Using the fact that the dual ball is balanced, that is, closed under
multiplication by scalars of absolute value less than or equal to $1$,
we can reflect about the origin to get that $\myfunc$ is upper semicontinuous.
So, $\myfunc$ is continuous on the dual ball.
It is hence the restriction of a continuous affine function on the
whole dual space;
see~\cite[Cor.~I.1.9]{alfsen_compact_convex_sets_and_boundary_integrals}.

So, from now on, assume that $u(x) \vee v(x) = \infty$,
for all $x\in\compactum$

Since $\myfunc$ is affine on the dual ball and infinite outside it,
the set where $\myfunc$ is finite is an extreme set of the dual ball.
Note that, given any distinct points $\mu_1$ and $\mu_2$ in the dual ball
such that $||\mu_2||<1$, there is a line segment in the dual ball having
$\mu_1$ as an endpoint and $\mu_2$ as a point in its relative interior.
It follows that if $\myfunc$ is finite at some point $\mu_2$ with $||\mu_2||<1$,
then $\myfunc$ is finite everywhere in the dual ball.
But this contradicts what we have just assumed, and we conclude that
$\myfunc(\mu)$ takes the value $+\infty$ if $||\mu||<1$.

\newcommand\bndymeas{\overline \mu}

Now, let $\mu$ be in the dual ball such that $||\mu||=1$.
By Choquet theory, there is a probability measure $\bndymeas$ on the dual ball
that has barycenter $\mu$ and is pseudo-concentrated on the extreme points
of the dual ball.
In the present case, since the set of extreme points is closed,
and hence measurable, $\bndymeas$ is concentrated on the
extreme points.

In fact, we have the following description of $\bndymeas$:
writing an arbitrary measurable subset $U$ of $\partial_e$ in the form
\begin{align*}
U = \{\delta_x \mid x\in U^+ \subset \compactum\}
      \union \{-\delta_x \mid x\in U^- \subset \compactum\},
\end{align*}
we have $\bndymeas[U] := \mu^+[U^+] + \mu^-[U^-]$.

By Lemma~\ref{lem:CK_approximate_from_below}, there is a non-decreasing net
$g_\alpha$ of continuous affine functions on the dual space
converging pointwise to $\myfunc$ on the dual ball.
Applying Theorem~\ref{thm:CK_barycenter_formula} and
Lemma~\ref{lem:CK_monotone_convergence_for_nets}, we get
\begin{align*}
\myfunc(\mu)
   &= \lim_\alpha g_\alpha(\mu) \\
   &= \lim_\alpha \int g_\alpha \,\dee\bndymeas \\
   &= \int \myfunc \,\dee\bndymeas \\
   &= \int_{\compactum} \myfunc(\delta_x) \,\dee\mu^+
        + \int_{\compactum} \myfunc(-\delta_x) \,\dee\mu^- \\
   &= \int_{\compactum} v(x) \,\dee\mu^+
        + \int_{\compactum} u(x) \,\dee\mu^-.
\qedhere
\end{align*}
\end{proof}

\begin{proof}[Proof of Proposition~\ref{prop:CK_characterise_affines}]
Any function $\Xi$ of the given form is clearly infinite outside the dual ball
and has infimum zero.
The rest was proved in Lemmas~\ref{lem:dual_is_lsc}, \ref{lem:dual_is_affine},
and~\ref{lem:CK_properties_imply_form}.
\end{proof}

\begin{lemma}
\label{lem:CK_dual_of_busemanns}
The function $\Xi$ is the Legendre--Fenchel transform of the function $\Phi$
in Theorem~\ref{thm:CK_characterise_busemanns}.
\end{lemma}

\begin{proof}
Fix $g\in\CK$,
and let $\Psi \colon \signedmeasures(\compactum) \to [0,\infty]$ be defined by
$\Psi(\mu) := \dotprod{\mu}{g} - \Xi(\mu)$.
By Lemmas~\ref{lem:dual_is_lsc} and~\ref{lem:dual_is_affine},
$\Psi$ is upper-semicontinuous and affine on the
unit ball %$\{\mu\in\ca \mid ||\mu||\le 1\}$
of $\ca$.
Outside the unit ball, $\Psi$ takes the value $-\infty$.
So, the supremum of $\Psi$ is attained at an extreme point of the unit
ball. The set of these extreme points is
$\{\delta_x \mid x\in\compactum\} \union \{-\delta_x \mid x\in\compactum\}$.
For all $x\in\compactum$, we have that $\Psi(\delta_x) = g(x) - v(x)$
and $\Psi(-\delta_x) = -g(x) - u(x)$.
It follows that the Legendre-Fenchel transform of $\Xi$ is
the function $\Phi$ given in~(\ref{eqn:busemanns_of_CK}).

Since $\Xi$ is a lower-semicontinuous proper convex function,
it is equal to the transform of its transform.
\end{proof}

We can now prove Theorem~\ref{thm:CK_characterise_busemanns}.

\begin{proof}[Proof of Theorem~\ref{thm:CK_characterise_busemanns}]
We combine Theorem~\ref{thm:characterise_busemen},
Proposition~\ref{prop:CK_characterise_affines}, and
Lemma~\ref{lem:CK_dual_of_busemanns}.
\end{proof}

\section{The horofunction boundary of the reverse-Funk geometry}
\label{sec:reverse_funk}

Although they are not strictly speaking metric spaces,
the reverse-Funk and Funk geometries retain enough of the properties of
metric spaces for the definition of the horofunction boundary, and of
Busemann points, to make sense.
In this and the following section, we study the boundary of these two
geometries.

Recall that the indicator function $\indicator_E$ of a set $E$ is defined
to take the value $1$ on $E$ and the value $0$ everywhere else.

\begin{lemma}
\label{lem:bounded_ratios}
Let $D$ be a compact convex subset of a locally-convex Hausdorff space $E$,
and let $f_1$ and $f_2$ be upper-semicontinuous affine functions on~$D$
with values in $[0,\infty)$.
If $\sup_D f_1 / g \le \sup_D f_2 / g$ for each continuous real-valued
affine function $g$ on $E$ that is positive on $D$, then $f_1 \le f_2$ on $D$.
\end{lemma}

\begin{proof}
Let $y$ be an extreme point of $D$. The function $1/\indicator_{\{y\}}$,
which takes the value $1$ at $y$ and the value $\infty$ everywhere else, is a
weak*-lower-semicontinuous affine function on $D$.
Therefore, there exists a non-decreasing net $g_\alpha$ of continuous
real-valued affine functions that are positive on $D$
such that $g_\alpha$ converges pointwise on $D$ to $1/\indicator_{\{y\}}$.
So, both $f_1/g_\alpha$ and $f_2/g_\alpha$
are non-increasing nets of real-valued upper-semicontinuous functions
on $D$ converging pointwise, respectively, to
$f_1\indicator_{\{y\}}$ and $f_2\indicator_{\{y\}}$.
By Lemma~\ref{lem:improved_convergence_of_sups},
$\sup_D f_1 / g_\alpha$ and $\sup_D f_2 / g_\alpha$ converge respectively
to $f_1(y)$ and $f_2(y)$.
We conclude that $f_1(y) \le f_2(y)$.

This is true for any extreme point of $D$, and the conclusion follows
upon applying Choquet theory.
\end{proof}

The proof of the next lemma is similar to that of the previous one.

\begin{lemma}
\label{lem:inverted_bounded_ratios}
Let $D$ be a compact convex subset of  a locally-convex Hausdorff space $E$,
and let $f_1$ and $f_2$ be lower-semicontinuous affine functions on $D$
with values in $(0,\infty]$.
If $\sup_D g / f_1 \le \sup_D g / f_2 $ for each continuous real-valued
affine function $g$ on $E$ that is positive on $D$, then $f_1 \ge f_2$ on $D$.
\end{lemma}

Suppose we have a reverse-Funk geometry on $C$, the interior of the cone of
an order unit space.
We consider the following cross-section of the dual cone:
$D := \{ y \in C^* \mid \dotprod{y}{b} = 1 \}$.
Observe that if $g$ is the restriction to $D$ of an element of $z$ of $C$,
and is normalised to have supremum $1$,
then $g$ is a continuous positive affine function and
\begin{align*}
\drevfunk(\cdot, z) - \drevfunk(b, z)
    = \log \sup_{y\in D} \frac{g(y)}{\dotprod{y}{\cdot}}.
\end{align*}

The following theorem gives the Busemann points of the reverse-Funk geometry.

\begin{theorem}
\label{thm:reverse_Funk__boundary}
Let $C$ be the interior of the cone of an order unit space,
and denote by $D$ the cross-section of the dual cone, as above.
The Busemann points of the reverse-Funk geometry on $C$ are the functions of
the following form:
\begin{align}
\label{eqn:reverse_funk_form}
\xi(x) := \log \sup_{y\in D} \frac{g(y)}{\dotprod{y}{x}},
\qquad\text{for all $x\in C$},
\end{align}
where $g$ is a weak*-upper-semicontinuous non-negative affine function on $D$
with supremum $1$,
that is not the restriction to $D$ of an element of $C$.
\end{theorem}

\begin{proof}

Let $\xi$ be of the above form.
Take a net $g_\alpha$ of elements of $C$ that, when viewed as a net of
continuous affine functions on $D$, is non-increasing and converges
pointwise to $g$. Fix $x\in C$.
So, the function $y \mapsto {g_\alpha(y)}/{\dotprod{y}{x}}$ defined on $D$
is non-increasing and converges pointwise to $g(y)/{\dotprod{y}{x}}$.
Therefore, by Lemma~\ref{lem:improved_convergence_of_sups}, the net
\begin{align*}
\drevfunk(x,g_\alpha) := \log \sup_{y\in D} \frac{g_\alpha(y)}{\dotprod{y}{x}}
\end{align*}
converges to $\xi(x)$.
In particular, $\drevfunk(b,g_\alpha)$ converges to zero.
It follows that $g_\alpha$ converges to $\xi$ in the reverse-Funk horofunction
compactification.

Moreover, the monotonicity of the convergence implies that
$\drevfunk(\cdot,g_\alpha) - \drevfunk(b,g_\alpha)$ is an almost non-increasing
net of functions; see the observation after
Definition~\ref{def:almost_non_increasing}.
Although Proposition~\ref{prop:almost_non_inc_busemann}
was stated for metric spaces, it also applies to the reverse-Funk geometry
since the only property used in the proof was the triangle inequality.
We conclude that $g_\alpha$ is an almost-geodesic,
and that $\xi$ is a Busemann point.

Now let $g_\alpha$ be an almost-geodesic net in $C$ converging to a Busemann
point $\xi$. So, $\drevfunk(\cdot,g_\alpha) - \drevfunk(b,g_\alpha)$
is an almost non-increasing net of functions converging to $\xi$.
By scaling $g_\alpha$ if necessary, we may assume that
$\drevfunk(b,g_\alpha)=0$, for all $\alpha$.

So, for any $\epsilon>0$, there exists an index $A$ such that
\begin{align*}
\sup_{y\in D} \frac{g_{\alpha'}(y)}{\dotprod{y}{x}}
    \le e^{\epsilon} \sup_{y\in D} \frac{g_\alpha(y)}{\dotprod{y}{x}},
\qquad\text{for all $x\in C$},
\end{align*}
whenever $\alpha$ and $\alpha'$ satisfy $A \le \alpha \le \alpha'$.
But this implies by Lemma~\ref{lem:bounded_ratios}
that $g_{\alpha'} \le e^{\epsilon} g_\alpha$ on $D$,
whenever $A \le \alpha \le \alpha'$.
We conclude that $\log g_\alpha|_D$ is an almost non-increasing net.

Applying Lemma~\ref{lem:improved_convergence_of_sups} and exponentiating,
we get that $g_\alpha$ converges pointwise on $D$ to an upper semicontinuous
function $g$, which is necessarily affine and non-negative.

By applying Lemma~\ref{lem:improved_convergence_of_sups} to the function
$\log g_{\alpha}(\cdot) - \log \dotprod{\cdot}{x}$ on $D$, we get that
$\xi$, which is the pointwise limit of $\drevfunk(\cdot,g_\alpha)$,
has the form given in the statement of the theorem.

The normalisation can be verified by evaluating at $b$.
Since it was assumed that $\xi$ is a Busemann point, and in particular a
horofunction, $g|_D$ can not be the restriction to $D$ of an element of $C$.
\end{proof}

\begin{theorem}
\label{thm:rfunk_detour_metric}
Let $\xi_1$ and $\xi_2$ be Busemann points of the reverse-Funk geometry,
corresponding via~(\ref{eqn:reverse_funk_form}) to affine functions
$g_1$ and $g_2$, respectively, with the properties specified in
Lemma~\ref{thm:reverse_Funk__boundary}.
Then, the distance between them in the detour metric is
\begin{align*}
\detourmetric(\xi_1, \xi_2)
   = \log\sup_{y\in D} \frac {g_1(y)} {g_2(y)}
          + \log\sup_{y\in D} \frac {g_2(y)} {g_1(y)}.
\end{align*}
(The supremum is always taken only over those points where the ratio
is well-defined).
\end{theorem}

\begin{proof}
For any $\lambda \in \R$, we have that $\xi_2 \le \xi_1 + \lambda$
if and only if
\begin{align*}
\sup_{y\in D} \frac{g_2(y)}{\dotprod{y}{x}}
    \le \sup_{y\in D} \frac{g_1(y)}{\dotprod{y}{x}} e^\lambda,
\qquad\text{for all $x\in C$}.
\end{align*}
By Lemma~\ref{lem:bounded_ratios},
this is equivalent to $g_2\le g_1 \exp(\lambda)$.

It follows using Proposition~\ref{prop:detour} that
\begin{align*}
H(\xi_1,\xi_2)
   = \inf\big\{\lambda\in\R \mid \xi_2(\cdot) \le \xi_1(\cdot) + \lambda \big\}
   = \log\sup_{y\in D} \frac{g_2(y)}{g_1(y)}.
\end{align*}
The result is now obtained upon symmetrising.
\end{proof}

\begin{corollary}
\label{cor:in_same_part_of_reverse_funk}
The two reverse-Funk Busemann points $\xi_1$ and $\xi_2$ are in the same part
if and only if $g_2 /\lambda \le g_1 \le \lambda g_2$, for some $\lambda>0$.
\end{corollary}

\newcommand\conv{\operatorname{conv}}

\begin{questions}
Is it possible for reverse-Funk geometries to have non-Busemann horofunctions?
This is not the case in finite dimension~\cite{walsh_hilbert},
and we will see in section~\ref{sec:positive_cone} that it is not the case
either for the positive cone $C^+(K)$.

The affine function $g$ in Theorem~\ref{thm:reverse_Funk__boundary}
can be extended in a unique way to a linear functional on the dual space.
One can calculate that the dual ball is $\conv(D\union -D)$.
Since $g$ is always bounded on $D$, its extension is continuous in the
dual-norm topology. 
What are the singleton Busemann points of the reverse-Funk geometry?
In finite dimension, there is a singleton corresponding to each extreme
ray of the closed cone $\overline C$; see~\cite{walsh_hilbert}.
In general, do the singletons correspond exactly to the extreme rays of the
bidual cone?
\end{questions}

We have some partial results concerning the singleton Busemann points of this
geometry.

\begin{proposition}
\label{prop:rfunk_x_rays_are_singletons}
Let $\xi$ be a Busemann point of the reverse-Funk geometry,
and let $g$ be as in~(\ref{eqn:reverse_funk_form}).
Extend $g$ to the whole of the dual space.
If $g$ is in an extreme ray of the bidual cone $C^{**}$,
then $\xi$ is a singleton Busemann point.
\end{proposition}

\begin{proof}
Let $\xi_1$ and $\xi_2$ be Busemann points in the same part.
By Theorem~\ref{thm:reverse_Funk__boundary}, we may write both of these points
in the form~(\ref{eqn:reverse_funk_form}), with $g_1$ and $g_2$,
respectively, substituted in for $g$.
Both $g_1$ and $g_2$ are bounded on $D$, and therefore 
their extensions to the whole dual space are continuous in the norm topology
of the dual.
Thus, they are both elements of the bidual.

By Corollary~\ref{cor:in_same_part_of_reverse_funk}, there exists
$\lambda>0$ such that $g_2 /\lambda \le g_1 \le  \lambda g_2$ on $D$.
Define $f = g_1 + g_2 / \lambda$ and $h = g_1 - g_2 / \lambda$.
Both $f$ and $h$ are linear functionals on the
dual space that are continuous in the dual-norm topology.
Moreover, they are non-negative on $D$. We conclude that $f$ and $h$ are in the
bidual cone.
But we have $g_1 = f/2 + h/2$,
which shows that $g_1$ is not in an extreme ray of the bidual cone.
\end{proof}

Let $U$ be the cone of
finite-valued weak$^*$-upper-semicontinuous
linear functionals on the dual space
that are non-negative on the dual cone $C^*$.

\begin{proposition}
\label{prop:rfunk_singletons_are_x_rays}
Let $\xi$ be a Busemann point of the reverse-Funk geometry,
and let $g$ be as in~(\ref{eqn:reverse_funk_form}).
Extend $g$ to the whole of the dual space.
If $\xi$ is a singleton, then $g$ is in an extreme ray of the cone $U$.
\end{proposition}

\begin{proof}
Suppose that $g = g_1 + g_2$, with $g_1$ and $g_2$ in $U$.
Let $g' := g_1 + 2 g_2$.
By normalising $g'$, the conditions of
Theorem~\ref{thm:reverse_Funk__boundary} are met, so we obtain a
Busemann point $\xi'$. Moreover, $g'/2 \le g \le g'$ on $D$, and so 
according to Theorem~\ref{thm:rfunk_detour_metric},
$\xi$ and $\xi'$ lie in the same part of the boundary.
So, by assumption, $\xi = \xi'$, which implies that $g'$ is a multiple of $g$,
which further implies that $g_1$ is a multiple of $g_2$.
\end{proof}

\section{The horofunction boundary of the Funk geometry}
\label{sec:funk}

The proof of the following theorem parallels that of the corresponding
result for the reverse-Funk geometry, Theorem~\ref{thm:reverse_Funk__boundary}.

\begin{theorem}
\label{thm:Funk__boundary}
The Busemann points of the Funk geometry on $C$
are the functions of the following form:
\begin{align}
\label{eqn:funk_form}
\xi(x) &:= \log \sup_{y\in D} \frac{\dotprod{y}{x}}{f(y)},
\qquad\text{for all $x\in C$},
\end{align}
where $f$ is a weak*-lower-semicontinuous non-negative affine function on $D$,
with infimum $1$, that is not the restriction to $D$ of an element of $C$.
\end{theorem}

The descriptions of the detour metric and the parts of the boundary
are also similar to the corresponding results for the reverse-Funk geometry.

\begin{theorem}
\label{thm:funk_detour_metric}
Let $\xi_1$ and $\xi_2$ be Busemann points of the Funk geometry,
corresponding via~(\ref{eqn:funk_form}) to affine functions
$f_1$ and $f_2$, respectively, with the properties specified in
Lemma~\ref{thm:Funk__boundary}.
Then, the distance between them in the detour metric is
\begin{align*}
\detourmetric(\xi_1, \xi_2)
   = \log\sup_{y\in D} \frac {f_1(y)} {f_2(y)}
          + \log\sup_{y\in D} \frac {f_2(y)} {f_1(y)}.
\end{align*}
(The supremum is always taken only over those points where the ratio
is well-defined).
\end{theorem}

\begin{corollary}
\label{cor:in_same_part_of_funk}
The two Funk Busemann points $\xi_1$ and $\xi_2$ are in the same part
if and only if $f_2 /\lambda \le f_1 \le \lambda f_2$, for some $\lambda>0$.
\end{corollary}

Unlike in the case of the reverse-Funk geometry, we can determine explicitly
the singleton Busemann points of the Funk geometry.
Recall that we have defined the cross section
$D := \{ y\in C^* \mid \dotprod{y}{b} = 1\}$.

\begin{corollary}
\label{cor:singletons_of_funk}
A function is a singleton Busemann point of the Funk geometry if and
only if it can be written $\log\dotprod{y}{\cdot}$,
where $y$ is an extreme point of $D$.
\end{corollary}

\begin{proof}
The proof is similar to that of
Corollary~\ref{cor:singletons_of_normed_space},
when one considers the cross-section $D$ instead of the dual ball.
\end{proof}

\section{The horofunction boundary of the Hilbert geometry}
\label{sec:hilbert}

In this section, we relate the boundary of the Hilbert geometry to those
of the reverse-Funk and Funk geometries.

\newcommand\Ctwo{\mathcal{C}}

We denote by $\projC$ the projective space of the cone $C$, and
by $[h]$ the projective class of an element $h$ of $C$.
Recall that we may regard the elements of $C$ as positive continuous
linear functionals on $C^*$.

\begin{proposition}
\label{prop:hilbert_almost_geo}
Let $z_\alpha$ be a net in $C$. Then, $z_\alpha$ is an
almost-geodesic in the Hilbert geometry if and only if it is an
almost-geodesic in both the Funk and reverse-Funk geometries.
\end{proposition}

\begin{proof}
For $\alpha$ and $\alpha'$ satisfying $\alpha\le\alpha'$, define
\begin{align*}
R(\alpha,\alpha') &:= 
    \drevfunk(b,z_{\alpha})
    + \drevfunk(z_{\alpha},z_{\alpha'})
    - \drevfunk(b,z_{\alpha'})
	   , \\
F(\alpha,\alpha') &:= 
     \dfunk(b,z_{\alpha})
    + \dfunk(z_{\alpha},z_{\alpha'})
    - \dfunk(b,z_{\alpha'})
	   ,
\qquad\text{and} \\
H(\alpha,\alpha') &:= 
    \dhil(b,z_{\alpha})
    + \dhil(z_{\alpha},z_{\alpha'})
    - \dhil(b,z_{\alpha'}) .
\end{align*}
Clearly $H = R + F$.
Also, by the triangle inequality, $R$, $F$, and $H$ are all non-negative.

For any $\alpha$ and $\alpha'$ with $\alpha\le\alpha'$, and any $\epsilon>0$,
we have that $H(\alpha,\alpha')<\epsilon$
implies $R(\alpha,\alpha')<\epsilon$ and $F(\alpha,\alpha')<\epsilon$.
Conversely, we have that $R(\alpha,\alpha')<\epsilon/2$ and
$F(\alpha,\alpha')<\epsilon/2$ implies $H(\alpha,\alpha')<\epsilon$.
The conclusion follows easily.
\end{proof}

We define the following compatibility relation between reverse-Funk
and Funk Busemann points.
We write $\xi_{R} \compat \xi_{F}$, when there exists a net in $C$
that is an almost-geodesic in both the reverse-Funk and the Funk geometry,
and converges to $\xi_{R}$ in the former and to $\xi_{F}$ in the latter.

\begin{theorem}
\label{thm:hilbert__boundary}
The Busemann points of the Hilbert geometry are the functions
of the form $\xi_H := \xi_{R} + \xi_{F}$, where $\xi_{R}$ and $\xi_{F}$ are,
respectively, reverse-Funk and Funk Busemann points,
satisfying $\xi_R \compat \xi_F$.
\end{theorem}

\begin{proof}
Let $\xi_{R}$ and $\xi_{F}$ be as in the statement.
So, there exists a net $z_\alpha$ in $C$ that is an almost-geodesic in
both the reverse-Funk geometry and the Funk geometry,
and converges to $\xi_{R}$ in the former and to $\xi_{F}$ in the latter.
Applying Proposition~\ref{prop:hilbert_almost_geo}, we get that $z_\alpha$
is an almost-geodesic in the Hilbert geometry, and it must necessarily
converge to $\xi_R + \xi_F$.

To prove the converse, let $\xi_H$ be a Busemann point of the Hilbert geometry,
and let $z_\alpha$ be an almost-geodesic net converging to it.
By Proposition~\ref{prop:hilbert_almost_geo}, $z_\alpha$ is an almost-geodesic
in both the Funk and reverse-Funk geometries, and so converges to a Busemann
point $\xi_F$ in the former and to a Busemann point $\xi_R$ in the latter.
So, $\xi_R \compat \xi_F$, and we also have that $\xi_H = \xi_R + \xi_F$.
\end{proof}

Our next theorem gives a formula for the detour metric in the Hilbert geometry.

\begin{theorem}
\label{thm:hilbert_detour_metric}
Let $\xi_H = \xi_R + \xi_F$ and $\xi'_H := \xi'_R + \xi'_F$
be Busemann points of a Hilbert geometry, each written as the
sum of a reverse-Funk Busemann point and a Funk Busemann point that
are compatible with one another.
Then, the distance between them in the detour metric is
\begin{align*}
\detourmetric_H(\xi_H, \xi'_H)
    = \detourmetric_R(\xi_R, \xi'_R) + \detourmetric_F(\xi_F, \xi'_F),
\end{align*}
where $\detourmetric_R$ and $\detourmetric_F$ denote, respectively, the
detour metrics in the reverse-Funk and Funk geometries.
\end{theorem}

\begin{proof}
Let $z_\alpha$ be a net in $C$ that is an almost-geodesic in both
the reverse-Funk and the Funk geometries, converging in the former to $\xi_R$
and in the latter to $\xi_F$.
By Proposition~\ref{prop:hilbert_almost_geo}, $z_\alpha$ is also an
almost-geodesic of the Hilbert geometry.
In this geometry, it converges to $\xi_H$.
By Lemma~\ref{lem:along_geos}, we have
\begin{align*}
H_R(\xi_R,\xi'_R)
    &= \lim_{\alpha} \big( d_R(b, z_\alpha) + \xi'_R(z_\alpha) \big)
\qquad\text{and} \\
H_F(\xi_F,\xi'_F)
    &= \lim_{\alpha} \big( d_F(b, z_\alpha) + \xi'_F(z_\alpha) \big),
\end{align*}
where $H_R$ and $H_F$ denote the detour cost in the reverse-Funk
and Funk geometries, respectively.
Adding, and using Lemma~\ref{lem:along_geos} again, we get that
$H_R(\xi_R, \xi'_R) + H_F(\xi_F, \xi'_F) = H_H(\xi_H, \xi'_H)$,
where $H_H$ is the Hilbert detour cost. The result follows upon symmetrising.
\end{proof}

\begin{theorem}
\label{thm:hilbert_unique_split}
Each Busemann point of the Hilbert geometry can be written in a unique way
as $\xi_H := \xi_{R} + \xi_{F}$, where $\xi_{R}$ and $\xi_{F}$ are,
respectively, reverse-Funk and Funk Busemann points,
satisfying $\xi_R \compat \xi_F$.
\end{theorem}

\begin{proof}
That each Busemann point can be written in this way was proved in 
Theorem~\ref{thm:hilbert__boundary}.

To prove uniqueness, suppose that $\xi_H = \xi_R + \xi_F = \xi'_R + \xi'_F$,
where $\xi_R$ and $\xi'_R$ are reverse-Funk Busemann points and $\xi_F$ and
$\xi'_F$ are Funk Busemann points, with $\xi_R \compat \xi_F$ and
$\xi'_R \compat \xi'_F$. 
By Theorem~\ref{thm:hilbert_detour_metric}, 
\begin{align*}
\detourmetric_R(\xi_R, \xi'_R) + \detourmetric_F(\xi_F, \xi'_F)
    = \detourmetric_H(\xi_H, \xi_H)
    = 0.
\end{align*}
It follows that both $\detourmetric_R(\xi_R, \xi'_R)$
and $\detourmetric_F(\xi_F, \xi'_F)$ are zero, and hence that
$\xi_R = \xi'_R$ and $\xi_F = \xi'_F$.
\end{proof}

Our next goal is to make explicit the meaning of the compatibility relation
$\compat$.

Suppose we are given two non-negative affine functions $f$ and $g$ on $D$
with the following properties.
We assume that $g$ is upper semicontinuous and has supremum $1$,
whereas $f$ is lower semicontinuous and has infimum $1$.
We assume further that $g$ takes the value zero everywhere that $f$ is finite.

Denote by $\Ctwo$ the set
\begin{align*}
\Ctwo := \big\{ (h,h') \in C \times C
                \mid \text{$[h] = [h']$ and $g < h \le h' < f$ on $D$} \big\}.
\end{align*}
We define a relation $\myord$ on $\Ctwo$ in the following way:
we say that $(h_1, h'_1) \myord (h_2, h'_2)$ if
$h_2 \le h_1$ and $h'_1 \le h'_2$ on $D$.
It is clear that this ordering is reflexive, transitive, and antisymmetric.

\newcommand\epigraph{\operatorname{epi}}
\newcommand\hypograph{\operatorname{hyp}}
\newcommand\newg{k}
\newcommand\sepfunc{l}

\begin{lemma}
\label{lem:supremum_of_Ctwo}
We have $f = \sup \{ h'|_D \mid (h, h') \in \Ctwo\}$
and $g = \inf \{ h|_D \mid (h, h') \in \Ctwo\}$.
\end{lemma}

\begin{proof}
Define the epigraph of $f$ and the (truncated) hypograph of $g$:
\begin{align*}
\epigraph f
    &:= \big\{ (x, \lambda) \in D \times \R \mid f(x) \le \lambda \big\}
\qquad\text{and} \\
\hypograph g
    &:= \big\{ (x, \lambda) \in D \times \R \mid g(x) \ge \lambda \ge 0 \big\}.
\end{align*}
Both of these sets are closed and convex, and $\hypograph g$ is compact.

Let $x\in D$, and let $\lambda < f(x)$.
Let $K$ be the convex hull of the union of
$\{(x, \lambda)\}$ and $\hypograph g$.
It is not hard to check that $K$ is compact and disjoint from
$\epigraph f$.
Therefore, by the Hahn--Banach separation theorem, there is a closed hyperplane
$H$ in $E\times\R$ that strongly separates $\epigraph f$ and $K$.
Here we are denoting by $E$ the affine hull of $D$.
Note that the strong separation implies that $H$ can not be of the form
$H' \times \R$, where $H'$ is a hyperplane of $E$.
It follows that $H$ is the graph of a continuous affine
function $h\colon E\to \R$, satisfying $g < h < f$ and $h(x) > \lambda$.

We can extend $h$ to the whole of the dual space in a unique way by
requiring homogeneity. Since $h$ is strictly positive on $D$,
this gives us an element of $C$, which we denote again by $h$.
So $(h, h)\in\Ctwo$.

Using that $\lambda$ can be chosen arbitrarily close to $f(x)$, we get that
$f(x) \le \sup \{ h'(x) \mid (h'', h') \in \Ctwo\}$.
The opposite inequality follows trivially from the definition of $\Ctwo$.

The second part is similar, but we must be careful because the epigraph of $f$
is not necessarily compact.
So, this time, we choose $\lambda$ arbitrarily so that
$\lambda> g(x)$, and separate $\hypograph g$ from the convex hull of
the following three sets,
\begin{align*}
\epigraph f \intersection \{(y,\beta) \in D\times \R \mid 0 \le \beta \le 2\},
\quad
\{(y,\beta) \in D\times \R \mid \beta = 2\},
\quad\text{and}\quad
\{(x, \lambda)\}.
\end{align*}
All three of these sets are compact, and, since none of them intersect
$\hypograph g$, neither does the convex hull of their union.

In the same manner as before, we obtain an element $h$ of $C$ satisfying
$g< h < \min(2, f)$ on $D$, and $h(x) < \lambda$, and the rest of the proof
is the same.
\end{proof}

\begin{lemma}
\label{lem:separate_functions}
Let $f$ be a lower-semicontinuous affine function on $D$, and let
$\{h_i\}_i$ be a finite collection of upper-semicontinuous affine functions
on $D$ satisfying $h_i<f$, for each $i$.
Then there exists a $h'\in C$ such that
$\max_i h_i < h' < f$ on $D$.
\end{lemma}

\begin{proof}
The proof is similar to that of the previous lemma.
We chose $I \in (-\infty, \inf f)$, and separate $\epigraph f$
from the convex hull of the union of the compact sets
\begin{align*}
\hypograph g_i \intersection
         \{(y,\beta) \in D \times \R \mid I \le \beta\},
\quad\text{for all $i$},
\quad\text{and }
\{(y,\beta) \in D\times \R \mid \beta = I\}.
\end{align*}
We obtain $h\in C$ satisfying $I < h< f$,
and $h(y) > g_i(y)$ for all $i$ and $y\in D$ such that $g_i(y) \ge I$.
The conclusion follows.
\end{proof}

\begin{lemma}
\label{lem:directed_set}
The set $\Ctwo$ is a directed set under the ordering $\myord$.
\end{lemma}

\begin{proof}
Let $(h_1, h'_1)$ and $(h_2, h'_2)$ be in $\Ctwo$.

By Lemma~\ref{lem:separate_functions}, there is a continuous real-valued
linear functional $h'$ satisfying
\begin{align*}
\max(h'_1, h'_2) < h' < f,
\qquad\text{on $D$}.
\end{align*}

Since $g$ is upper-semicontinuous, and $h_1$ and $h_2$ are continuous,
the function $\min(h_1, h_2) - g$ attains its infimum over $D$.
This infimum is positive.
Choose an $\epsilon\in(0,1)$ strictly smaller than this infimum.
Let $\lambda \in (0,1)$ be such that $0 < \lambda h' < \epsilon$ on $D$.
So, $\newg := g + \lambda h'$ is a non-negative
upper-semicontinuous affine function on $D$.
We have that
\begin{align*}
\max(g, \lambda h'_1, \lambda h'_2) < \newg < \min(h_1, h_2).
\end{align*}
Also, since $g$ takes the value zero everywhere that $f$ is finite,
we have $\newg < \lambda f$.

We deduce using Lemma~\ref{lem:separate_functions} that there exists a
real-valued continuous linear functional $\sepfunc$
satisfying $\newg < \sepfunc < \min(h_1, h_2, \lambda f)$ on $D$.
Hence $g < \sepfunc$.
Moreover,
\begin{align*}
\max(h'_1, h'_2) < \frac{\sepfunc}{\lambda} < f.
\end{align*}

We have thus proved that $(l, l/\lambda)$ is in $\Ctwo$,
and that
\begin{align*}
(h_1, h'_1) &\myord \Big(l, \frac{l}{\lambda}\Big)
\qquad\text{and}\qquad
(h_2, h'_2) \myord \Big(l, \frac{l}{\lambda}\Big).
\qedhere
\end{align*}
\end{proof}

We can now say which reverse-Funk and Funk Busemann points are compatible.

\begin{proposition}
\label{prop:compatibility}
Let $C$ be a cone giving rise to a complete Hilbert geometry.
Let $\xi_{R}$ and $\xi_{F}$ be, respectively, a reverse-Funk Busemann point
(of the form~(\ref{eqn:reverse_funk_form}) and a Funk Busemann point
(of the form~(\ref{eqn:funk_form}).
Then, $\xi_{R} \compat \xi_{F}$ if and only if,
for each $y\in D$, either $g(y)=0$ or $f(y)=\infty$.
\end{proposition}

\begin{proof}
Assume that $g$ and $f$ satisfy the stated condition.
The set $\Ctwo$ with the ordering $\myord$ defined using $g$ and $f$
is a directed set, by Lemma~\ref{lem:directed_set}.
Consider the net $z_\alpha$ defined on the directed set $\Ctwo$
by $z_\alpha := \alpha$, for all $\alpha\in\Ctwo$.
Write $(g_\alpha, f_\alpha) := z_\alpha$, for each $\alpha\in\Ctwo$.
Observe that $g_\alpha$ is non-increasing, and $f_\alpha$ is non-decreasing.

Combining Lemmas~\ref{lem:improved_convergence_of_sups}
and~\ref{lem:supremum_of_Ctwo},
we get that the net $f_\alpha$ converges to $f$.

Fix $x\in C$. So, the net of functions
$y \mapsto {\dotprod{y}{x}}/{f_\alpha(y)}$
is non-increasing and converges pointwise to ${\dotprod{y}{x}}/f(y)$.
Therefore, by Lemma~\ref{lem:improved_convergence_of_sups}, the net
\begin{align*}
\dfunk(x,f_\alpha) = \log \sup_{y\in D} \frac {\dotprod{y}{x}} {f_\alpha(y)}
\end{align*}
converges to $\xi_F(x)$.
In particular, $\dfunk(b,f_\alpha)$ converges to zero.
It follows that $f_\alpha$ converges to $\xi_F$ in the compactification
of the Funk geometry.
Moreover, the monotonicity of the convergence implies that
$\dfunk(\cdot,f_\alpha) - \dfunk(b,f_\alpha)$ is an almost non-increasing
net of functions (see the observation after
Definition~\ref{def:almost_non_increasing}).
So, by Proposition~\ref{prop:almost_non_inc_busemann},
$f_\alpha$ is an almost-geodesic.
(Note that, although this proposition was stated for metric spaces,
it also applies to the Funk geometry since all that was required in the
proof was the triangle inequality.)

Recall that convergence to a point in the horofunction boundary
of the Funk geometry is a property of the projective classes of the points
rather than of the points themselves.
So, $[f_\alpha]$ converges in the Funk geometry to the Funk Busemann point
$\xi_F$.

The same method works to show that $[g_\alpha]$ converges in the
reverse-Funk geometry to the reverse-Funk Busemann point $\xi_R$.
Recall, moreover, that $[f_\alpha]=[g_\alpha]$, for all $\alpha$.
We have shown that $\xi_{R} \compat \xi_{F}$.

To prove the converse, assume that $\xi_{R} \compat \xi_{F}$ holds.
So, there is a net $z_\alpha$ in $C$ that is an almost-geodesic
in both the Funk and reverse-Funk geometries, and converges to $\xi_F$
in the former and to $\xi_R$ in the latter.

By using reasoning similar to that in second part of the proof of
Theorem~\ref{thm:reverse_Funk__boundary},
we get that $z_\alpha / \exp(\drevfunk(b,z_\alpha))$
converges pointwise to $g$ on $D$.
Similarly, $z_\alpha \exp(\dfunk(b,z_\alpha))$ converges pointwise to $f$.

It follows that
$\dhil(b,z_\alpha) := \drevfunk(b,z_\alpha) + \dfunk(b,z_\alpha)$ converges
to $\log (f(y)/g(y))$, for all $y\in D$.
But this net grows without bound according
to Proposition~\ref{prop:almost_geos_are_infinite}, and so
the latter function is identically infinity.
We have shown that, at each point of $D$, either $g$
is zero, or $f$ is infinite.
\end{proof}

Next we show that compatibility between a reverse-Funk and a Funk
Busemann point only depends the respective part that each lies in.

\begin{proposition}
\label{prop:mix_and_match}
Assume the Hilbert geometry is complete.
Let $\xi_R$ and $\xi'_R$ be reverse-Funk Busemann points in the same part,
and let $\xi_F$ and $\xi'_F$ be Funk Busemann points in the same part.
If $\xi_R \compat \xi_F$, then $\xi'_R \compat \xi'_F$.
\end{proposition}

\begin{proof}
This follows from combining Proposition~\ref{prop:compatibility}
with Corollaries~\ref{cor:in_same_part_of_reverse_funk}
and~\ref{cor:in_same_part_of_funk}.
\end{proof}

\begin{corollary}
\label{cor:singletons_of_hilbert}
A Busemann point $\xi_H = \xi_R + \xi_F$ of a complete Hilbert geometry,
with $\xi_R \compat \xi_F$, is a singleton if and only if
$\xi_R$ and $\xi_F$ are singleton Busemann points of, respectively,
the reverse-Funk and Funk geometries.
\end{corollary}

\begin{proof}
Assume that $\xi_R$ is not a singleton, that is, there exists another
reverse-Funk Busemann point $\xi'_R$ in the same part as it.
By, Proposition~\ref{prop:mix_and_match}, $\xi'_R \compat \xi_F$,
and so, by Theorem~\ref{thm:hilbert__boundary},
$\xi'_H := \xi'_R + \xi_F$ is a Busemann point of the Hilbert geometry.
From Theorem~\ref{thm:hilbert_detour_metric}, we see that $\xi'_H$ and $\xi_H$
lie in the same part. Hence, $\xi_H$ is not a singleton.
One may also prove in the same way that if $\xi_F$ is not a singleton,
then neither is $\xi_H$.

Assume now that there exists a Busemann point $\xi'_H = \xi'_R + \xi'_F$,
of the Hilbert geometry, with $\xi'_R \compat \xi'_F$,
that is distinct from $\xi_H$ but in the same part as it.
So, either $\xi_R$ and $\xi'_R$ are distinct, or $\xi_F$ and $\xi'_F$ are.
By Theorem~\ref{thm:hilbert_detour_metric}, $\xi'_R$ is in the same part
as $\xi_R$, and $\xi'_F$ is in the same part as $\xi_F$.
This shows that either $\xi_R$ or $\xi_F$ is not a singleton.
\end{proof}

\section{The Hilbert geometry on the cone $\uncontpos$}
\label{sec:positive_cone}

In this section, we study the positive cone $\uncontpos$, that is,
the cone of positive continuous functions on a compact Hausdorff space $\unint$.
We take the basepoint $b$ to be the function that is identically equal to~$1$.
The dual cone of $\uncontpos$ is the cone $\signedmeasures^+(\compactum)$
of regular Borel measures on $\compactum$.
The cross section
$D := \{\mu \in \signedmeasures^+(\compactum) \mid \langle \mu, b \rangle = 1\}$
consists of the probability measures on $K$.
The extreme points of this cross section are the Dirac masses.

\subsection{The boundary of the reverse-Funk geometry on $\uncontpos$}

The reverse-Funk metric on $\uncontpos$ is given by
\begin{align*}
\drevfunk(f, g) =  \log \sup_{x\in\unint} \frac{g(x)}{f(x)},
\qquad\text{for all $g$, $f$ in $\uncontpos$.}
\end{align*}

Recall that the hypograph of a function $f \colon  X\to[-\infty,\infty]$ is the
set $\hypograph f := \{(x,\alpha)\in X\times\R \mid \alpha \le f(x)\}$.
A net of functions is said to converge in the hypograph topology if their
hypographs converge in the Kuratowski--Painlev\'e topology.

\begin{proposition}
\label{prop:rev_funk_positive_cone}
The horofunctions of the reverse-Funk geometry on the positive cone $\uncontpos$
are the functions of the form
\begin{align}
\label{eqn:positive_reverse_funk}
\xi_R (h) := \log \sup_{x\in\unint} \frac{g(x)}{h(x)},
\qquad\text{for all $h\in\uncontpos$},
\end{align}
where $g \colon \unint\to[0,1]$ is an upper-semicontinuous function
with supremum $1$, but is not both positive and continuous.
All these horofunctions are Busemann points.
\end{proposition}

\begin{proof}
Let $\xi_R$ be of the above form.
Write
\begin{align*}
\bar g(\mu) := \int_\compactum g\,\dee\mu,
\qquad\text{for all $\mu\in \signedmeasures^+(\compactum)$}.
\end{align*}
Let $h\in\uncontpos$, and denote by $D_h$ the set of elements $\mu$ of
the dual cone such that $\dotprod{\mu}{h} := \int_K h\,\dee\mu = 1$.
So $D_h$ is a cross-section of the dual cone.
Observe that $\bar g$ is an upper-semicontinuous affine function
on $D_h$, and so attains its supremum over this set at one of the
extreme points of the set.
These extreme points are the points of the form $\delta_x/h(x)$,
with $x\in K$, where $\delta_x$ denotes the unit atomic mass at a point $x$.
Using that the function $\bar g(\mu)/\dotprod{\mu}{h}$ is invariant under
scaling $\mu$, we get
\begin{align*}
\sup_{\mu\in D} \frac{\bar g(\mu)}{\dotprod{\mu}{h}}
    = \sup_{\mu\in D_h} \bar g(\mu)
    = \sup_{x\in K} \frac{g(x)}{h(x)}.
\end{align*}
Note that, on $D$, the function $\bar g$ is non-negative, upper-semicontinuous, and affine, and that its supremum over this set is $1$.
We conclude, using Theorem~\ref{thm:reverse_Funk__boundary},
that the function $\xi_R$ is a Busemann point of the reverse-Funk geometry
of $\uncontpos$.

Now let $g_\alpha$ be a net in $\uncontpos$ converging to a horofunction.
By scaling if necessary, we may assume that the supremum of each
$g_\alpha$ is $1$.
Consider the hypographs $\hypograph(g_\alpha)$ of these functions.
This is a net of closed subsets
of $\unint\times\R$.
By the theorem of Mrowka~(see~\cite[Theorem 5.2.11, page 149]{beer_book}),
this net has a subnet that converges in the Kuratowski--Painlev\'e topology.
Therefore, $g_\alpha$ has a subnet that converges in the hypograph topology
to a proper upper-semicontinuous function $g$.
We reuse the notation $g_\alpha$ to denote this subnet.

Since $g_\alpha$ takes values in $[0,1]$, so also does $g$.
From~\cite[Theorem~5.3.6, page 160]{beer_book}, we get that
$\sup g_\alpha$ converges to $\sup g$. Thus, $\sup g=1$.

Fix $h\in\uncontpos$.
We have that $g_\alpha/h$ converges in the hypograph
topology to the proper upper semicontinuous function $g/h$.
Applying~\cite[Theorem~5.3.6]{beer_book} again, we get that
\begin{align*}
\lim_\alpha \sup_{x\in\unint} \frac{g_\alpha(x)}{h(x)}
   = \sup_{x\in\unint} \frac{g(x)}{h(x)}.
\end{align*}
Since $h$ was chosen arbitrarily, we see that $g_\alpha$ converges
in the horofunction boundary to a point of the required form.
\end{proof}

\subsection{The boundary of the Funk geometry on $\uncontpos$}

\begin{proposition}
\label{prop:funk_positive_cone}
The horofunctions of the Funk geometry on the positive cone $\uncontpos$
are the functions of the form
\begin{align}
\label{eqn:positive_funk}
\xi_F (h) := \log \sup_{x\in\unint} \frac{h(x)}{f(x)},
\qquad\text{for all $h\in\uncontpos$},
\end{align}
where $f \colon \unint\to[1,\infty]$ is a lower-semicontinuous function
with infimum $1$, but is not both finite and continuous.
All these horofunctions are Busemann points.
\end{proposition}

\begin{proof}
Use Proposition~\ref{prop:rev_funk_positive_cone} and the fact that
the pointwise reciprocal map is an isometry taking the reverse-Funk
metric to the Funk metric.
\end{proof}

\subsection{The boundary of the Hilbert geometry on $\uncontpos$}

\begin{proposition}
\label{prop:hilbert_positive_cone}
The Busemann points of the Hilbert geometry on the positive cone $\uncontpos$
are the functions of the form
\begin{align*}
\xi_H (h)
    := \log \sup_{x\in\unint} \frac{g(x)}{h(x)}
         + \log \sup_{x\in\unint} \frac{h(x)}{f(x)},
\qquad\text{for all $h\in\uncontpos$},
\end{align*}
where $g \colon \unint\to[0, 1]$ is an upper-semicontinuous function
with supremum $1$, and $f \colon \unint\to[1,\infty]$ is a lower-semicontinuous
function with infimum $1$, and,
for each $x\in\compactum$, either $g(x)=0$ or $f(x)=\infty$.
\end{proposition}

\begin{proof}
By Theorem~\ref{thm:hilbert__boundary}, the Busemann points of the Hilbert
geometry are exactly the functions of the form $\xi_R + \xi_F$, with
$\xi_R\compat\xi_F$, where $\xi_R$ and $\xi_F$ are Busemann points of,
respectively, the reverse-Funk and Funk geometries.
The Busemann points of these geometries were described in
Propositions~\ref{prop:rev_funk_positive_cone}
and~\ref{prop:funk_positive_cone}.
Let $g$ and $f$ be as in those propositions, and write
$\bar g(\mu) := \int_\compactum g\,\dee\mu$, and
$\bar f(\mu) := \int_\compactum f\,\dee\mu$, for all $\mu\in D$.
Proposition~\ref{prop:compatibility} states that $\xi_R\compat\xi_F$
if and only if $\bar g$ and $\bar f$ are not both positive and finite 
at any point of $D$.
It is not too hard to show that this condition is equivalent to
$g$ and $f$ not being both positive and finite at any point of $K$.
\end{proof}

We have seen that all reverse-Funk horofunctions and all Funk horofunctions
on the cone $\uncontpos$ are Busemann points.
However, it is not necessarily true that all Hilbert horofunctions on this cone
are Busemann. Indeed, consider the case where $K := [0,1]$ and take for
example $g := \indicator_{[0,1/2)} / 2 + \indicator_{[1/2, 1]}$ and
$f := \indicator_{[0,1/2]} + 2 \indicator_{(1/2, 1]}$.
Here $\indicator_E$ is the indicator function of a set $E$, which
takes the value $1$ on $E$ and the value $0$ everywhere else.
By the propositions above, $\xi_R$ is a Busemann point of the reverse-Funk
geometry and $\xi_F$ is a Busemann point of the Funk geometry,
where $\xi_R$ and $\xi_F$ are defined as in~(\ref{eqn:positive_reverse_funk})
and~(\ref{eqn:positive_funk}), respectively.
Observe that if $h_n$ is a non-increasing sequence of continuous
functions on $\uncontpos$ that converges pointwise to $g$, then
$h_n$ is an almost-geodesic and converges to $g$ in the reverse-Funk geometry;
see Figure~\ref{fig:example_simultaneously}.
Moreover, it converges to $\xi_F$ in the Funk geometry, although it is not an
almost geodesic in this geometry.
This shows that $\xi_R + \xi_F$ is a horofunction of the Hilbert geometry.
However, $\xi_R + \xi_F$ is not a Busemann of this geometry,
according to Proposition~\ref{prop:hilbert_positive_cone},
since $f$ and $g$ do not satisfy the compatibility condition.

Thus, the situation differs from the finite-dimensional case.
There, the reverse-Funk horofunctions are all automatically Busemann,
and every Hilbert horofunction is Busemann if and only if every
Funk horofunctions is; see~\cite{walsh_hilbert}.

\begin{figure}
\centering
\input{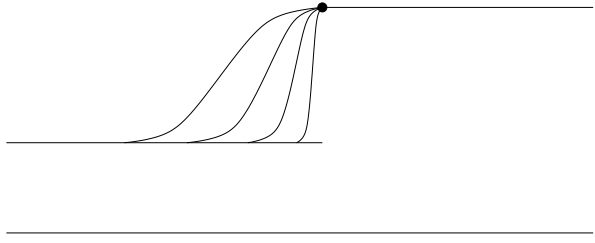}
\caption{A sequence converging to a horofunction}
\label{fig:example_simultaneously}
\end{figure}

\section{The horofunction boundary of the Thompson geometry}
\label{sec:thompson}

Here we study the boundary of the Thompson geometry.

Recall that reverse-Funk Busemann points are of the
form~(\ref{eqn:reverse_funk_form}) and Funk Busemann points
are of the form~(\ref{eqn:funk_form}).
It was shown in Proposition~\ref{prop:compatibility}, that if
$\xi_{R}$ and $\xi_{F}$ are Busemann points of their respective geometries,
then $\xi_{R} \compat \xi_{F}$ if and only if,
for each $y\in D$, either $g(y)=0$ or $f(y)=\infty$.
Here, $g$ and $f$ are the functionals appearing
in~(\ref{eqn:reverse_funk_form}) and~(\ref{eqn:funk_form}).

For each $x\in C$, define the following functions on $C$:
\begin{align*}
r_{x}(\cdot) := \log\frac{\gauge{}{x}{\cdot}}{\gauge{}{x}{b}}
\qquad\text{and}\quad
f_{x}(\cdot) := \log\frac{\gauge{}{\cdot}{x}}{\gauge{}{b}{x}}.
\end{align*}

Let $\vee$ and $\wedge$ denote, respectively, maximum and minimum.
We use the convention that addition and subtraction take
precedence over these operators.
We write $x^+:=x\vee 0$ and $x^-:=x\wedge 0$.
Let $\Raug:=\R\union\{-\infty,+\infty\}$.
Given two real-valued functions $f_1$ and $f_2$ on the same set,
and $c\in\Raug$, define
\begin{align*}
[f_1,f_2,c] := f_1 + c^- \vee f_2 - c^+.
\end{align*}
Observe that if $c=\infty$, then $[f_1,f_2,c] = f_1$,
whereas if $c=-\infty$, then $[f_1,f_2,c] = f_2$.

Let $B_{R}$ and $B_{F}$ be the set of Busemann points of the reverse-Funk
and Funk geometries, respectively.

\begin{proposition}
\label{prop:busemann_thompson}
Let $C$ be a cone giving rise to a complete Thompson geometry.
The set of Busemann points of this geometry is
\begin{align*}
\{r_z \mid z\in C\}
   \union \{f_{z} \mid z\in C\}
   \union B_{R}
   \union B_{F}
   \union\{ [\xi_R,\xi_F,c]
       \mid \xi_R \in B_{R}, \xi_F\in B_{F}, \xi_R \compat \xi_F, c\in \R \}.
\end{align*}
\end{proposition}

\begin{proof}
Let $\xi_T$ be a Busemann point of the Thompson geometry, and let $z_\alpha$ be
an almost-geodesic net in $C$ converging to it.

By taking subnets if necessary, we may assume that $z_\alpha$ converges
in both the Funk and reverse-Funk horofunction compactifications,
to limits $\xi_F$ and $\xi_R$, respectively,
and furthermore that $\drevfunk(b,z_\alpha) - \dfunk(b,z_\alpha)$
converges to a limit $c$ in $\Raug$.
So, as $\alpha$ tends to infinity,
\begin{align*}
\tmetric(b, z_\alpha) - \drevfunk(b, z_\alpha) &\to -c^-,
\quad\text{and} \\
\tmetric(b, z_\alpha) - \dfunk(b, z_\alpha) &\to c^+.
\end{align*}
Therefore,
\begin{align*}
\tmetric(y, z_\alpha) - \tmetric(b, z_\alpha)
    &= \Big ( \drevfunk(y, z_\alpha) \vee \dfunk(y, z_\alpha) \Big )
          - \tmetric(b, z_\alpha) \\
    &= \Big ( \drevfunk(y, z_\alpha) - \drevfunk(b, z_\alpha) + c^- \Big )
           \vee \Big ( \dfunk(y, z_\alpha) - \dfunk(b, z_\alpha) - c^+ \Big ) \\
    &\to [\xi_R, \xi_{F}, c] (y).
\end{align*}

Consider the case where $c < \infty$.
Let $\lambda>0$ be such that $c < 2\log\lambda$.
For $\alpha$ large enough,
\begin{align*}
\drevfunk(\lambda b,z_\alpha) - \dfunk(\lambda b,z_\alpha)
   = \drevfunk(b,z_\alpha) - \dfunk(b,z_\alpha) - 2\log\lambda
   < 0,
\end{align*}
and hence $\tmetric(\lambda b,z_\alpha) = \dfunk(\lambda b,z_\alpha)$.
Note also that
$\tmetric(z_\alpha, z_{\alpha'}) \ge \dfunk(z_\alpha, z_{\alpha'})$,
for all $\alpha$ and $\alpha'$.

Recall that an almost-geodesic remains an almost-geodesic when the basepoint is
changed. It will be convenient to consider almost-geodesics with respect to the
basepoint $\lambda b$.

Let $\epsilon>0$ be given.
Since $z_\alpha$ is an almost geodesic in the Thompson geometry,
we have, for $\alpha$ and $\alpha'$ large enough, with $\alpha\le\alpha'$,
\begin{align*}
\tmetric(\lambda b, z_{\alpha'})
    \ge \tmetric(\lambda b, z_\alpha)
       + \tmetric(z_\alpha, z_{\alpha'}) - \epsilon,
\end{align*}
and hence, again for $\alpha$ and $\alpha'$ large enough, with $\alpha\le\alpha'$,
\begin{align*}
\dfunk(\lambda b, z_{\alpha'})
    \ge \dfunk(\lambda b, z_\alpha) + \dfunk(z_\alpha, z_{\alpha'}) - \epsilon.
\end{align*}
We deduce that $z_\alpha$ is an almost-geodesic in the Funk geometry,
and so $\xi_F$ is either of the form $f_z$, with $z\in C$,
or a Funk Busemann point.

Similarly, when $c > -\infty$,
$z_\alpha$ is an almost-geodesic in the reverse-Funk geometry and $\xi_R$
is either of the form $r_z$, with $z\in C$, or a reverse-Funk Busemann point.

So, if $c=\infty$, then $\xi_T$ is in $\{r_z \mid z\in C\} \union B_{R}$.
On the other hand, if $c=-\infty$,
then $\xi_T$ is in $\{f_{z} \mid z\in C\} \union B_{F}$.

There remains the case where $c$ is finite. Since $\xi_T$ was assumed to be
in the horofunction boundary, we have,
by Proposition~\ref{prop:almost_geos_are_infinite},
that $\tmetric(b, z_{\alpha})$ converges to infinity, and so
both $\drevfunk(b, z_{\alpha})$ and $\dfunk(b, z_{\alpha})$ do too,
since their difference remains bounded.
Therefore both $\xi_R$ and $\xi_F$ are Busemann points.
We have shown that $\xi_T = [\xi_R, \xi_F, c]$, with
$\xi_R \in B_{R}$ and $\xi_F\in B_{F}$ such that $\xi_R \compat \xi_F$.
\end{proof}

We extended the definition of $H$ by setting
$H(\xi+u,\eta +v):=H(\xi,\eta)+v-u$ for all Busemann points $\xi$ and $\eta$,
and $u,v\in[-\infty,0]$. Here, we use the convention that $-\infty$
is absorbing for addition.
The following proposition was proved in~\cite{walsh_gauge} in the
finite-dimensional case, but the proof carries over to infinite dimension.

\begin{proposition}
\label{prop:detour_thompson}
The distance in the detour metric between two Busemann points $\xi_T$
and $\xi'_T$ in a complete Thompson geometry is
$\delta(\xi_T,\xi'_T) = \dhil(x,x')$
if $\xi_T=r_x$ and $\xi'_T=r_{x'}$, with $x,x'\in C$.
The same formula holds when $\xi_T=f_{x}$ and $\xi'_T=f_{x'}$,
with $x,x'\in C$.
If $\xi_T=[\xi_R,\xi_F,c]$ and $\xi'_T=[\xi'_R,\xi'_F,c']$,
with $\xi_R, \xi'_R \in B_{R}$, $\xi_F, \xi'_F\in B_{F}$,
$\xi_R \compat \xi_F$, $\xi'_R \compat \xi'_F$, $c, c'\in \Raug$, then
\begin{align*}
\delta(\xi_T,\xi'_T)=
   \max\Big( H(\bar \xi_R,\bar \xi'_R), H(\bar \xi_F, \bar \xi'_F) \Big)
   + \max\Big( H(\bar \xi'_R,\bar \xi_R), H(\bar \xi'_F, \bar \xi_F) \Big),
\end{align*}
where
\begin{align*}
\bar \xi_R &:= \xi_R + c^- , &\bar \xi_F &:= \xi_F - c^+ , \\
\bar \xi'_R &:= \xi'_R + c'^- , &\bar \xi'_F &:= \xi'_F - c'^+ .
\end{align*}
In all other cases, $\delta(\xi_T,\xi'_T)=\infty$.
\end{proposition}

\begin{corollary}
\label{cor:thompson_boundary}
The set of singletons of a complete Thompson geometry is exactly the union
of the Funk singletons and the reverse-Funk singletons.
\end{corollary}

\section{Thompson geometries isometric to Banach spaces}
\label{sec:thompson_isometries}

\newcommand\norm{||\cdot||}
\newcommand\mymax{\vee}
\newcommand\mysum{\oplus_n}
\newcommand\image{\operatorname{Im}}

In this section we determine which Thompson geometries are isometric to
Banach spaces.

We start with a technical lemma.

\begin{lemma}
\label{lem:monotonicity}
For all $\alpha,\beta\in\R$, the sequence
$p_n := n^{-1}\log(\exp(n\alpha) + \exp(n\beta))$ is non-increasing.
\end{lemma}

\begin{proof}
Fix $n\in\N$.
Observe that, for any $r\in\{0,\dots,n\}$,
\begin{align*}
e^{(\alpha-\beta)(n-r)} + e^{(\beta-\alpha)r} \ge 1,
\end{align*}
since one or other of the terms is greater than or equal to $1$.
Equivalently, 
\begin{align*}
e^{\alpha n} + e^{\beta n} \ge e^{\alpha r} e^{\beta (n-r)}.
\end{align*}
By considering binomial coefficients and using the previous inequality,
we get that
\begin{align*}
(e^{\alpha n} + e^{\beta n}) (e^{\alpha n} + e^{\beta n})^n
    \ge \big(e^{\alpha (n+1)} + e^{\beta (n+1)}\big)^n
\end{align*}
Taking logarithms and rearranging, we get $p_n\ge p_{n+1}$.
\end{proof}

Recall that a linear subspace of a Riesz space (vector lattice) $E$
is a \emph{Riesz subspace} if it is closed under the lattice operations on $E$.
We will need the lattice version of the Stone--Weierstrass theorem.
This theorem states~\cite{aliprantis_burkinshaw_principles_of_real_analysis}
that if $K$ is a compact space,
then any Reisz subspace of $C(K)$ that separates the points of $K$ and
contains the constant function $1$ is uniformly dense in $C(K)$.

The setting for the next lemma is an order unit space $(V, \overline C, u)$.
Recall that the Thompson metric is defined on $C$,
the interior of $\overline C$ with respect to the topology on $V$
coming from the order unit norm.

\begin{lemma}
\label{lem:new_properties_of_map_into_CK}
Let $\Phi$ be a bijection from a linear space $X$ to $C$.
Take the basepoint to be $b:=\Phi(0)$,
and let $K$ be the pointwise closure of some set of Funk singletons.
Assume that the pullback $f\after \Phi$ of each element $f$ of $K$ is linear.
Then, the map $\phi \colon  X \to C(K)$, $x\mapsto \phi_x$, where
\begin{align*}
\phi_x(f) := f(\Phi(x)),
\qquad\text{for all $f \in K$},
\end{align*}
is linear and its image is dense in $C(K)$.
\end{lemma}

\begin{proof}
Observe that, as a closed subset of a compact set, $K$ is compact.

We have, for all $\alpha,\beta\in\R$, and $x,y\in X$, and $f\in K$,
\begin{align*}
\phi_{\alpha x+\beta y}(f)
    = f\after\Phi (\alpha x+\beta y)
    = \alpha f\after\Phi(x) + \beta f\after\Phi(y)
    = \alpha \phi_x (f) + \beta \phi_y(f).
\end{align*}
Therefore, $\phi$ is linear.
Furthermore, the image of $\phi$ is a linear subspace of $C(K)$.

For each $n\in\N$, define the following operation on $X$:
\begin{align*}
x \mysum y := \frac{1}{n} \Phi^{-1}(\Phi(nx) + \Phi(ny)),
\qquad\text{for $x, y \in X$}.
\end{align*}
Fix $x$ and $y$ in $X$.
Let $f\in K$, and write $g:=f\after \Phi$.
So, $g$ is a linear functional on $X$.
Recall that, according to Corollary~\ref{cor:singletons_of_funk},
$f$ is the logarithm of a linear functional $l$ on the cone,
that is, $f={\log}\after l$ for some linear functional $l$ on $C$.
Note that $l={\exp}\after f$. We have, for each $n\in\N$,
\begin{align*}
g(x \mysum y)
    &= \frac{1}{n} g\big(\Phi^{-1}(\Phi(nx) + \Phi(ny))\big) \\
    &= \frac{1}{n} f(\Phi(nx) + \Phi(ny)) \\
    &= \frac{1}{n} \log(l\after\Phi(nx) + l\after\Phi(ny)) \\
    &= \frac{1}{n} \log\big(\exp\after f\after\Phi(nx)
                               + \exp\after f\after\Phi(ny)\big) \\
    &= \frac{1}{n} \log\big(\exp(ng(x)) + \exp(ng(y))\big).
\end{align*}
Using Lemma~\ref{lem:monotonicity}, we get that the sequence
$g(x \mysum y) = \phi_{x \mysum y}(f)$ converges monotonically to its
limit, $\max\{g(x), g(y)\}$, as $n$ tends to infinity.
Since this is true for every $f\in K$,
we have that $\phi_{x \mysum y}$ converges monotonically and pointwise to
$\phi_x \vee \phi_y$, and hence converges to this limit uniformly,
by Dini's theorem.
Therefore, $\phi_x \vee \phi_y$ is in $\closure\image\phi$,
the closure of the image of $\phi$.
It follows that $\closure\image\phi$ is a Reisz subspace of $C(K)$.

Let $f_1$ and $f_2$ be distinct elements of $K$.
So, there is some $y\in C$ such that $f_1(y)\neq f_2(y)$.
Setting $x:=\Phi^{-1}(y)$, we can write this $\phi_x(f_1) \neq \phi_x(f_2)$.
This shows that $\image\phi$ separates the points of $K$.

Recall that we have chosen the basepoint $b$ of the cone
so that $b=\Phi(0)$. Write $x:=\Phi^{-1} (e b)$, where $e$ is Euler's number.
Let $f\in K$. Since $f$ is the logarithm of a linear functional,
$f(\alpha z) = \log\alpha + f(z)$ for all points $z$ in the cone,
and all $\alpha>0$.
So, since $f(b)=0$, we get $f(eb) = 1$.
Hence, $\phi_x(f) = f(eb) = 1$.
We have shown that the constant function $1$ is in the image of $\phi$.

Applying the Stone--Weierstrass theorem to $\closure\image\phi$,
we get that $\closure\image\phi$ is dense in $C(K)$,
and so $\image\phi$ is dense in $C(K)$.
\end{proof}

\begin{theorem}
\label{thm:thompson_banach}
If a Thompson geometry is isometric to a Banach space, then the cone
is linearly isomorphic to $\uncontpos$, for some compact Hausdorff space $K$.
\end{theorem}
 
\begin{proof}
Assume that $\Phi \colon  X \to C$ is an isometry from a Banach space
$(X,\norm)$ to a Thompson geometry on a cone $C$.
We choose the basepoint $b$ of the cone so that $b=\Phi(0)$.

Let $K$ be the pointwise closure of the set of Funk singletons of $C$.
Each element $f$ of $K$ can be written $f = \log y|_C$, where
$y$ is in the weak* closure of the set of extreme points of the cross-section
$D := \{y\in C^* \mid \dotprod{y}{b} = 1\}$ of the dual cone $C^*$.
Since it is a closed subset of a compact set, $K$ is compact.

By Corollary~\ref{cor:thompson_boundary}, each Funk singleton $f$
is also a Thompson singleton, and so its pullback $f\after \Phi$
is a singleton Busemann point of the Banach space $X$,
and is therefore linear, by Corollary~\ref{cor:singletons_of_normed_space}.
It follows that $f\after \Phi$ is linear for all $f\in K$.

Let $\phi$ be defined as in Lemma~\ref{lem:new_properties_of_map_into_CK}.
According to that lemma, the image of the Banach space $X$ under $\phi$
is a uniformly dense subspace of $C(K)$.

Let $F$ denote the set of Funk singletons.
Recall that the following formula for the Funk metric holds for an arbitrary
cone (see~\cite[Proposition~4.4]{walsh_gauge}):
\begin{align*}
\dfunk(w,z) = \sup_{f\in F}\big( f(w) - f(z) \big),
\qquad\text{for all $w,z\in C$}.
\end{align*}
So, the Thompson metric is given by
\begin{align*}
\thompson(w,z)
    = \dfunk(w,z) \vee \dfunk(z,w)
    = \sup_{f\in F}\big| f(w) - f(z) \big|,
\qquad\text{for all $w,z\in C$}.
\end{align*}
The same formula holds with $F$ replaced by $K$ since the former is dense in
the latter.
We conclude that, for all $x,y\in X$,
\begin{align*}
||y-x||
    = \thompson(\Phi(x),\Phi(y))
    = \sup_{f\in K}| f(\Phi(x)) - f(\Phi(y)) |
    = \sup_{f\in K} |\phi_x(f) - \phi_y(f)|.
\end{align*}
Therefore $\phi$ is an isometry from $(X, ||\cdot||)$ to $C(K)$ with the
supremum norm.

But we have assumed that $X$ is complete, and so its image under $\phi$
is complete.
We conclude that this image is the whole of $C(K)$.

We have shown that $\phi$ is an isometric linear-isomorphism from $X$ to $C(K)$.

Define the map $\Theta := \exp \after \phi \after \Phi^{-1}$
from $C$ to $C^+(K)$.
For each $p\in C$, we have
\begin{align*}
(\Theta p)(f)
    = \exp \after \phi_{\Phi^{-1} p}(f)
    = e^{f(p)},
\quad\text{for all $f\in K$}.
\end{align*}
Since $f$ is the logarithm of a linear functional on $C$, it follows
that $\Theta$ is linear.
So, $\Theta$ is a linear isomorphism between $C$ and $\uncontpos$.
\end{proof}

\section{Hilbert geometries isometric to Banach spaces}
\label{sec:hilbert_isometries}

Here we prove that the only Hilbert geometries isometric to Banach spaces
are the ones on the cones $\uncontpos$, for some compact space $K$.

We first require some lemmas concerning singleton Busemann points in
cone geometries and in Banach spaces.

\begin{lemma}
\label{lem:to_infinity_and_beyond}
Let $z_\alpha$ be an almost-geodesic net in the Funk geometry on a cone $C$,
converging to a Funk Busemann point $\xi_F$
and normalised so that $\dfunk(b, z_\alpha) = 0$ for all $\alpha$.
If $\xi_R$ is a reverse-Funk Busemann point such that
$\xi_R(z_\alpha)$ converges to $-\infty$, then $\xi_R\compat \xi_F$.
\end{lemma}

\begin{proof}
Write $\xi_{R}$ and $\xi_{F}$ in the form (\ref{eqn:reverse_funk_form})
and~(\ref{eqn:funk_form}), respectively, with appropriate $g$ and~$f$.

Choose $\epsilon>0$. Since $z_\alpha$ is a Funk almost-geodesic,
by Proposition~\ref{prop:almost_non_inc_busemann}
$\dfunk(\cdot, z_\alpha) - \dfunk(b, z_\alpha)$ is almost non-increasing.
Note that the proposition was stated for metric spaces, but it also
applies to the Funk geometry since the only property used was the
triangle inequality.
So, for $\alpha$ large enough,
$\dfunk(x, z_\alpha) > \xi_F(x) - \epsilon$, for all $x\in C$.
This is equivalent to
\begin{align*}
\sup_{y\in D}\frac{\dotprod{y}{x}}{\dotprod{y}{z_\alpha}}
    > e^{-\epsilon}
\sup_{y\in D}\frac{\dotprod{y}{x}}{f(y)},
\qquad\text{for all $x\in C$},
\end{align*}
where $D := \{y \in C^* \mid \dotprod{y}{b} = 1 \}$.
Applying Lemma~\ref{lem:inverted_bounded_ratios},
we get that $\dotprod{\cdot}{z_\alpha} \le \exp(\epsilon)f$ on $D$,
for $\alpha$ large enough.

Since $\xi_R(z_\alpha)$ converges to $-\infty$, we have that
$\sup_{y\in D} g(y) / \dotprod{y}{z_\alpha}$ converges to zero,
which implies that $g(y) / \dotprod{y}{z_\alpha}$ converges to zero
for all $y\in D$.
We deduce that, if $g(y)$ is non-zero for some $y\in D$,
then $\dotprod{y}{z_\alpha}$ converges to infinity, and so $f(y)$ is infinite.
The conclusion follows on applying Proposition~\ref{prop:compatibility}.
\end{proof}

\begin{lemma}
\label{lem:translations_dont_move_singletons}
Let $s$ be a singleton of a Banach space,
and let $\phi \colon  x\mapsto x+v$ be the translation by some vector $v$.
Then, $\phi$ leaves $s$ invariant.
\end{lemma}

\begin{proof}
By Corollary~\ref{cor:singletons_of_normed_space}, $s$ is linear,
and the conclusion follows easily.
\end{proof}

\begin{lemma}
\label{lem:one_way_or_another}
Let $s$ and $s'$ be singletons of a Banach space $(X, ||\cdot||)$.
Then, there exists an almost-geodesic net $z_\alpha$ converging to $s$
such that
$s'(z_\alpha)$ converges to either $\infty$ or $-\infty$.
\end{lemma}

\begin{proof}
Let $z_\alpha$ be an almost geodesic converging to $s$,
and denote by $\directed$ be the directed set on which the net $z_\alpha$
is based.
By taking a subnet if necessary, we may assume that $s'(z_\alpha)$
converges to a limit in $[-\infty, \infty]$. If this limit is
infinite, then the conclusion of the lemma holds, so assume the contrary.

Take a point $v$ in the Banach space such that $s'(v) < 0$.
Observe that for each $n\in\N$, the net $z_\alpha + nv$ is an almost
geodesic, and by Lemma~\ref{lem:translations_dont_move_singletons}
it converges to $s$.

We denote by $d(\cdot, \cdot)$ the metric coming from the norm $||\cdot||$.

Denote by $\neighbourhoods$ the set of neighbourhoods of $s$ in the
horofunction compactification of $X$.
Let $\newdirected$ be the set of elements $(W, x, n, \epsilon, A)$ of
$\neighbourhoods \times X \times \N \times (0,\infty) \times \directed$
such that
\renewcommand{\theenumi}{(\roman{enumi})}
\renewcommand{\labelenumi}{\theenumi}
\begin{enumerate}
\item
\label{enum:first}
$x = z_\alpha + nv$ for some $\alpha \ge A$;
\item
\label{enum:second}
$x\in W$;
\item
\label{enum:third}
$0 \le d(b, x) + s(x) < \epsilon$.
\end{enumerate}

Note that, given any $W \in\neighbourhoods$, $n\in\N$, $\epsilon\in(0,\infty)$,
and $A\in\directed$, we can find an $x\in X$ satisfying~\ref{enum:first},
\ref{enum:second}, and~\ref{enum:third}.

Define on $\newdirected$ an order by
$(W, x, n, \epsilon, A) \le (W', x', n', \epsilon', A')$
if the two elements are the same,
or if $W\supset W'$, $n\le n'$, $\epsilon \ge \epsilon'$, $A\le A'$, and
\begin{align}
\label{eqn:converged_at_x}
|d(x, y) - d(b, y) - s(x)| < \epsilon,
\qquad\text{for all $y\in W'$}.
\end{align}
It is not hard to verify that this order makes $\newdirected$
into a directed set.

Define the net $y_\beta := x$, for $\beta := (W, x, n, \epsilon, A)$.
It is clear that $y_\beta$ converges to $s$.

Suppose some $\lambda>0$ is given. If $\beta := (W, x, n, \epsilon, A)$
is large enough that $\epsilon < \lambda/2$, and
$\beta' :=(W', x', n', \epsilon', A')$ is such that $\beta\le\beta'$,
then combining~\ref{enum:third} and~(\ref{eqn:converged_at_x}) we get
\begin{align*}
d(b, y_\beta) + d(y_\beta, y_{\beta'}) - d(b, y_{\beta'}) < \lambda.
\end{align*}
This proves that the net $y_\beta$ is an almost geodesic.

Since $s'$ is linear, $s'(y_\beta) = s'(z_\alpha) + ns'(v)$, for all $\beta$,
where $\alpha$ depends on $\beta$ and is as in \ref{enum:first}.
Both $\alpha$ and $n$ can be made as large as we wish by taking $\beta$
large enough, and so $s'(y_\beta)$ converges to $-\infty$.
\end{proof}

Observe that if $\xi_R$ and $\xi_F$ are reverse-Funk and Funk horofunctions,
respectively, then $\xi_R + \xi_F$ is constant on each projective class of the cone,
and so we may consider it to be defined on $\projC$.
Recall that, according to Corollary~\ref{cor:singletons_of_hilbert},
the singleton Busemann points of a complete Hilbert geometry are the 
points of the form $\xi_R + \xi_F$,
where $\xi_R$ and $\xi_F$ are singleton Busemann points of, respectively,
the reverse-Funk and Funk geometries, satisfying $\xi_R \compat \xi_F$.

\begin{lemma}
\label{lem:similar_to_ones_own_reflection}
Let $\Phi \colon X \to \projC$ be an isometry from a Banach space to the
Hilbert geometry on a cone.
Let $s$ be a singleton of the Banach space.
Write $s = (\xi_R + \xi_F) \after \Phi$
and $-s = (\xi_R' + \xi_F') \after \Phi$,
where $\xi_R$ and $\xi_R'$ are singletons of the reverse-Funk geometry,
and $\xi_F$ and $\xi_F'$ are singletons of the Funk geometry,
satisfying $\xi_R\compat\xi_F$ and $\xi_R'\compat\xi_F'$.
Then, $\xi_R = -\xi_F'$ and $\xi_R' = -\xi_F$.
\end{lemma}

\begin{proof}
Let $[p]$ and $[q]$ be points in the projective space $\projC$ of the cone,
and fix representatives, $p$ and $q$.

Let $z_\beta$ be an almost geodesic in the Banach space converging to $s$.
Taking the reflection $z'_\beta := 2\Phi^{-1}([p])-z_\beta$
in the point $\Phi^{-1}([p])$, we get an almost geodesic converging to $-s$.

Let $[y_\beta] := \Phi(z_\beta)$, for all $\beta$,
and take respresentatives $y_\beta$.
So, $y_\beta$ is an almost geodesic in the Hilbert geometry,
and therefore it is also
an almost geodesic in both the reverse-Funk and Funk geometries.
Hence, it converges to a Busemann point in both of these geometries.
Moreover, the sum of the two Busemann points is equal to the Hilbert geometry
Busemann point $\xi_R+\xi_F$. Since, by Theorem~\ref{thm:hilbert_unique_split},
a Hilbert geometry Busemann point can be written in a unique way as the sum
of a reverse-Funk Busemann point and a Funk Busemann point, we see that
the limit of $y_\beta$ in the reverse-Funk geometry is $\xi_R$ and that the
limit of $y_\beta$ in the Funk geometry is $\xi_F$.

Likewise, let $[y'_\beta] := \Phi(z'_\beta)$, for all $\beta$, and take
representatives $y'_\beta$.
Again, this is an almost geodesic in the Hilbert geometry,
converging this time to $\xi_R'$ in the reverse-Funk geometry
and to $\xi_F'$ in the Funk geometry.

For all $\beta$,
since $\Phi^{-1}([p])$ is the midpoint of $z'_\beta$ and $z_\beta$,
we have
\begin{align*}
\dhil(y'_\beta, p) + \dhil(p, y_\beta) = \dhil(y'_\beta, y_\beta).
\end{align*}
This implies that 
\begin{align}
\label{eqn:dead_straight}
\dfunk(y'_\beta, p) + \dfunk(p, y_\beta) &= \dfunk(y'_\beta, y_\beta),
\qquad\text{for all $\beta$}.
\end{align}
Recall that
\begin{align*}
\xi_F(q) - \xi_F(p) &= \lim_\beta \big( \dfunk(q, y_\beta) - \dfunk(p, y_\beta) \big)
\qquad\text{and} \\
\xi_R'(q) - \xi_R'(p)
   &= \lim_\beta \big( \dfunk(y'_\beta, q) - \dfunk(y'_\beta, p) \big).
\end{align*}
Combining these two equations with~(\ref{eqn:dead_straight}), and using the
Funk metric triangle inequality applied to points $y'_\beta$,
$q$, and $y_\beta$, we get $\xi_F(q) - \xi_F(p) + \xi_R'(q) - \xi_R'(p) \ge 0$.

But this holds for arbitrary $p$ and $q$ in $C$,
and so $\xi_R' + \xi_F$ must be constant on $C$.
Since this function takes the value zero at the basepoint $b$,
we see that it is zero everywhere.

The proof that $\xi_R + \xi_F' = 0$ is similar.
\end{proof}

Given a cone $C$, we use $F$ to denote the set of singleton Busemann points
$\xi_F$ of the Funk geometry on $C$ such that there exists a reverse-Funk
singleton Busemann point $\xi_R$ satisfying $\xi_R\compat \xi_F$.

\newcommand\startfunk{\eta_F}
\newcommand\startrev{\eta_R}
\newcommand\starthil{\eta_H}

\begin{lemma}
\label{lem:all_linear}
Let $C$ be a cone whose Hilbert geometry is isometric to a Banach space $X$.
Let $\Phi' \colon X' \to C$ be a bijection from another linear space $X'$
to $C$, such that the pullback $\xi_H\after \Phi'$ of every Hilbert geometry
singleton $\xi_H$ is a linear functional on $X'$.
If there exists $\startfunk \in F$ whose pullback is linear,
then the pullback of every element of $F$ is linear.
\end{lemma}

\begin{proof}
Since $\startfunk$ is in $F$, there is some reverse-Funk singleton $\startrev$
satisfying $\startrev \compat \startfunk$,
and hence $\starthil := \startrev + \startfunk$ is a Hilbert singleton,
by Corollary~\ref{cor:singletons_of_hilbert}.

Let $\xi_F\in F$. So, there is some singleton $\xi_R$ of the reverse-Funk
geometry
such that $\xi_H := \xi_R + \xi_F$ is a singleton of the Hilbert geometry.
From the isometry between the Banach space $X$ and the Hilbert geometry, we get
that there is another singleton $\xi_H'$ of the Hilbert geometry satisfying
$\xi_H' = -\xi_H$. We may write $\xi_H' = \xi_R' + \xi_F'$, where $\xi_R'$ and $\xi_F'$ are
singletons of the reverse-Funk and Funk geometries, respectively.

Using Lemma~\ref{lem:one_way_or_another} and the isometry between the Banach
space $X$ and the Hilbert geometry, we get that there exists an almost
geodesic net $[y_\alpha]$ in $\projC$ converging in the Hilbert geometry
to $\starthil$ such that either $\xi_H(y_\alpha)$ or $\xi_H'(y_\alpha)$
converges to $-\infty$.

Consider the former case.
For all $\alpha$, take a representative $y_\alpha$ of the projective class
$[y_\alpha]$ so that $\dfunk(b,y_\alpha) = 0$.
This implies, for each $\alpha$,
that $\dotprod{\cdot}{b} \le \dotprod{\cdot}{y_\alpha}$ on the cross-section
$D$ of $C^*$,
from which it follows that $\dfunk(b, \cdot) \le \dfunk(y_\alpha, \cdot)$
on $C$, and hence that $\xi_F(y_\alpha) \ge 0$. Since we are supposing that
$\xi_H(y_\alpha) = \xi_R(y_\alpha) + \xi_F(y_\alpha)$ converges to $-\infty$,
we must have that $\xi_R(y_\alpha)$ converges to $-\infty$.

From Proposition~\ref{prop:hilbert_almost_geo}, we get that $y_\alpha$
is an almost geodesic in the Funk geometry. In this geometry, it converges
to $\startfunk$.
Applying Lemma~\ref{lem:to_infinity_and_beyond},
we get that $\xi_R\compat \startfunk$.

This means that $\xi_R + \startfunk$ is a singleton of the Hilbert geometry
on $C$, and so, by assumption, its pullback $(\xi_R + \startfunk) \after \Phi'$
is linear. We deduce that $\xi_R \after \Phi'$ is linear, and hence that
$\xi_F\after \Phi' = (\xi_H - \xi_R) \after \Phi'$ is linear.

Now consider the latter case, that is,
where $\xi_H'(y_\alpha) = \xi_R'(y_\alpha) + \xi_F'(y_\alpha)$ converges to $-\infty$.
Using reasoning similar to that in the previous case, we get that
$\xi_R' \after \Phi'$ is linear.
But, according to Lemma~\ref{lem:similar_to_ones_own_reflection},
$\xi_R' = -\xi_F$, and so $\xi_F\after\Phi'$ is linear.
\end{proof}

\begin{lemma}
\label{lem:singleton_form}
If a Hilbert geometry on a cone $C$ is isometric to a Banach space,
then the singleton Busemann points of the Hilbert geometry are exactly the
functions of the form $\xi_F - \xi_F'$, with $\xi_F$ and $\xi_F'$
distinct elements of $F$.
\end{lemma}

\begin{proof}
Let $h$ be a Hilbert geometry singleton.
By Corollary~\ref{cor:singletons_of_hilbert},
we may write $\xi_H = \xi_R + \xi_F$,
where $\xi_R$ and $\xi_F$ are reverse-Funk and Funk singletons, respectively,
satisfying $\xi_R\compat \xi_F$. So, $\xi_F$ is in $F$.
From Lemma~\ref{lem:similar_to_ones_own_reflection}, we get that
$\xi_R = -\xi_F'$, for some $\xi_F'\in F$. This shows that $\xi_H$ is of the
required form.

Let $\xi_F$ and $\xi_F'$ be distinct elements of $F$.
By Corollary~\ref{cor:singletons_of_funk},
we may write $\xi_F = \log\dotprod{y}{\cdot}$
and $\xi_F' = \log\dotprod{y'}{\cdot}$,
where $y$ and $y'$ are distinct extreme points of the cross-section $D$
of the dual cone.

Let $y_\alpha$ be a net in $C$ such that $\dotprod{\cdot}{y_\alpha}$
is non-decreasing and converges pointwise on $D$ to $1/\indicator_{y}$.
By Lemma~\ref{lem:improved_convergence_of_sups}, $\dfunk(b, y_\alpha)$
converges to $0$.
Write $z_\alpha := y_\alpha \exp(\dfunk(b, y_\alpha))$, for all~$\alpha$.
So, $\dfunk(b, z_\alpha) = 0$, for all $\alpha$,
and $z_\alpha$ is an almost geodesic converging to $\xi_F$ in the Funk geometry.

Observe that $\xi_F'(z_\alpha) = \log\dotprod{y'}{z_\alpha}$
converges to $\infty$, as $\alpha$ tends to infinity.
But, by Lemma~\ref{lem:similar_to_ones_own_reflection}, $-\xi_F' = \xi_R$ for
some reverse-Funk singleton $\xi_R$.
Applying Lemma~\ref{lem:to_infinity_and_beyond},
we get that $\xi_R\compat \xi_F$.
So, by Corollary~\ref{cor:singletons_of_hilbert},
$\xi_F - \xi_F'$ is a Hilbert geometry singleton.
\end{proof}

Let $K$ be a compact space.
Define the following seminorm on $C(K)$.
\begin{align*}
||x||_H := \sup_{f\in K} x(f) - \inf_{f\in K} x(f).
\end{align*}
Denote by ${\equiv}$ the equivalence relation on $C(K)$ where two functions
are equivalent if they differ by a constant, that it, $x \equiv y$
if $x = y + c$ for some constant $c\in\R$.
The seminorm $||x||_H$ is a norm on the quotient $C(K) / {\equiv}$.
This space is a Banach space, and we denote it by $H(K)$.

\begin{theorem}
\label{thm:hilbert_banach}
If a Hilbert geometry on a cone $C$ is isometric to a Banach space,
then $C$ is linearly isomorphic to $\uncontpos$ for some compact
Hausdorff space $K$.
\end{theorem}

\begin{proof}
Let $\Phi \colon X \to \projC$ be an isometry from a Banach space $(X,\norm)$
to the Hilbert geometry on $C$.

Each singleton of the Hilbert geometry may be written $\xi_R+\xi_F$,
where $\xi_R$ and $\xi_F$ are singletons of the reverse-Funk
and Funk geometries, respectively, and $\xi_R \compat \xi_F$.
Let $F$ denote the set of Funk singletons that appear in this way,
and denote by $K$ the pointwise closure of this set of functions.

Recall that each element $f$ of $K$ can be written $f = \log y|_C$, where
$y$ is in the weak* closure of the set of extreme points of the cross-section
$D := \{y\in C^* \mid \dotprod{y}{b} = 1\}$ of the dual cone $C^*$.
Since $K$ is a closed subset of a compact set, it is compact.

Consider the linear space $X' := X \times \R$, and
define a map $\Phi' \colon  X' \to C$ in the following way.
Fix a choice of a particular $\startfunk\in F$.
For each $x\in X$, the projective class $\Phi(x)$ is a ray in $C$, and along
this ray the function $\startfunk$ is monotonically increasing, taking values from
$-\infty$ to $\infty$. So, for each $x\in X$ and $\alpha\in\R$,
we may define $\Phi'(x, \alpha)$ to be the representative
of $\Phi(x)$ where $\startfunk$ takes the value $\alpha$.

Observe that, since by definition
\begin{align}
\label{eqn:def_of_Phi}
(\startfunk \after \Phi')(x, \alpha) = \alpha,
\qquad\text{for all $x\in X$ and $\alpha\in\R$},
\end{align}
we have that $\startfunk \after \Phi'$ is linear.

Consider a Hilbert singleton $\xi_H$.
Since Hilbert horofunctions are constant on projective classes,
$\xi_H\after \Phi'(x, \alpha) = \xi_H\after\Phi(x)$.
But $\xi_H\after\Phi$ is a singleton of $X$, and hence linear.
We conclude that the pullback $\xi_H\after \Phi'$ of $\xi_H$ is linear on $X'$.

Applying Lemma~\ref{lem:all_linear}, we get that $f\after \Phi'$ is linear
for all $f\in F$, and it follows that the same is true for all $f\in K$.

Define the map $\phi \colon  X' \to C(K)$, $x\mapsto \phi_x$ as in
Lemma~\ref{lem:new_properties_of_map_into_CK}. That is,
\begin{align*}
\phi_x(f) := f(\Phi'(x)),
\qquad\text{for all $f \in K$}.
\end{align*}
According to
Lemma~\ref{lem:new_properties_of_map_into_CK},
$\phi$ is linear and its image is a uniformly dense subspace of $C(K)$.

We make the following claim, the proof of which we postpone.

\begin{claim}
\label{claim:complete_image}
The image of $X'$ under the map $\phi$ is a complete subset of
$(C(K), ||\cdot||_\infty)$.
\end{claim}

From this claim, we conclude that the image of $\phi$ is the whole of $C(K)$.

Extend the norm on $X$ to a seminorm on $X'$ by ignoring the second coordinate.
We again use $||\cdot||$ to denote this seminorm.
We make a second claim.

\begin{claim}
\label{claim:hilbert_isometric_to_CK}
The map $\phi$ is an isometry from $(X', ||\cdot||)$ to $(C(K), ||\cdot||_H)$.
\end{claim}

It follows, upon quotienting on each side by the respective
$1$-dimensional subspace where the seminorm is zero, that $X$ is isometric to
$H(K)$.

Define the map $\Theta := \exp \after \phi \after \Phi'^{-1}$
from $C$ to $\uncontpos$.
For each $p\in C$, we have
\begin{align*}
(\Theta p)(f)
    = \exp \after \phi_{\Phi'^{-1} (p)}(f)
    = e^{f(p)},
\quad\text{for all $f\in K$}.
\end{align*}
Since $f$ is the logarithm of a linear functional on $C$, it follows
that $\Theta$ is linear.
So, $\Theta$ is a linear isomorphism between $C$ and $\uncontpos$.
\end{proof}

\begin{proof}[Proof of Claim~\ref{claim:hilbert_isometric_to_CK}]
Let $S$ be the set of singletons of the Banach space $(X, ||\cdot||)$.
Extend each $s\in S$ to $X'$ by $s(x, \alpha):= s(x)$,
for all $x\in X$ and $\alpha\in\R$.

We have that $||x|| = \sup_{s\in S} s(x)$, for all $x$ in $X$.
From Lemma~\ref{lem:singleton_form}, a function $s$ is in $S$
if and only if it is of the form $s = \xi_F \after \Phi' - \xi_F' \after \Phi'$,
with $\xi_F$ and $\xi_F'$ distinct elements of $F$.
Note that if $\xi_F$ and $\xi_F'$ were identical,
then $\xi_F \after \Phi' - \xi_F' \after \Phi'$ would be zero.
So, for any $p := (x, \alpha) \in X'$,
\begin{align*}
||p||
    &= \sup_{\xi_F, \xi_F'\in F}
          \big( \xi_F \after \Phi'(p) - \xi_F' \after \Phi'(p) \big) \\
    &= \sup_{\xi_F\in F} \phi_p(\xi_F) - \inf_{\xi_F'\in F} \phi_p(\xi_F') \\
    &= || \phi_p ||_H,
\end{align*}
since $F$ is dense in $K$.
\end{proof}

\begin{proof}[Proof of Claim~\ref{claim:complete_image}]
Recall that we extend the norm on $X$ to a seminorm $||\cdot||$ on $X'$
by ignoring the second coordinate.
Consider another norm on $X'$ defined by
\begin{align*}
||p||' := ||x|| + |\alpha|,
\qquad\text{for all $p:=(x, \alpha)$ in $X'$}.
\end{align*}
This is the $\ell_1$-product of two complete norms, and so makes $X'$ into
a Banach space.

From Claim~\ref{claim:hilbert_isometric_to_CK} and~(\ref{eqn:def_of_Phi}),
we see that $\phi$ induces an isometry between the norm $||\cdot||'$ on $X'$
and the norm $||\cdot||''$ on $C(K)$ defined by
\begin{align*}
||g||'' := ||g||_H + |g(\startfunk)|,
\qquad\text{for all $g$ in $C(K)$}.
\end{align*}
But it is not hard to show that
$||\cdot||_\infty \le ||\cdot||'' \le 3||\cdot||_\infty$,
and so $||\cdot||''$ and $||\cdot||_\infty$ are equivalent norms on $C(K)$.
It follows that the image of $\phi$ is a complete subset of $(C(K), ||\cdot||_\infty)$.
\end{proof}

\bibliographystyle{plain}
\bibliography{infinite}

\end{document}